\documentclass[
final
]{dmtcs-episciences}

\usepackage[table]{xcolor}
\usepackage[utf8]{inputenc}

\usepackage{amsmath,amssymb,amsthm}

\newtheorem{theorem}{Theorem}
\newtheorem{lemma}[theorem]{Lemma}
\newtheorem{proposition}[theorem]{Proposition}
\newtheorem{observation}[theorem]{Observation}

\theoremstyle{definition}
\newtheorem{example}[theorem]{Example}

\usepackage{cleveref}
\usepackage{graphicx}
\usepackage{esvect}
\usepackage[shortlabels]{enumitem}
\usepackage{diagbox}
\usepackage{todonotes}

\usepackage[round]{natbib}

\def\paren#1{\left( #1 \right)}
\def\acc#1{\left\{ #1 \right\}}
\def\floor#1{\left\lfloor #1 \right\rfloor}

\renewcommand{\le}{\leqslant}
\renewcommand{\ge}{\geqslant}
\newcommand{\normaltt}[1]{\normalfont\texttt{#1}}
\newcommand{\LL}{\mathcal L}
\DeclareMathOperator{\RT}{RT}

\author{L\!'ubom{\'i}ra Dvo\v{r}\'akov\'a\affiliationmark{1}
\and Lucas Mol\affiliationmark{2}\thanks{Research
of Lucas Mol is supported by NSERC grant RGPIN-2021-04084.}
\and Pascal Ochem\affiliationmark{3}\thanks{Research
of Pascal Ochem is supported by the ANR project CADO (ANR-24-CE48-3758-01).}
}
\title{Critical exponent of ternary words{ } \\ with few distinct palindromes}
\affiliation{
  FNSPE Czech Technical University, Prague, Czech Republic \\
  Department of Mathematics and Statistics, Thompson Rivers University, Kamloops, Canada \\
  LIRMM, CNRS, Universit\'e de Montpellier, Montpellier, France 
  }

\begin{document}
\publicationdata{vol. 28:2}{2026}{22}{10.46298/dmtcs.16125}{2025-07-29; 2025-07-29; 2026-03-13}{2026-03-19}
\maketitle

\begin{abstract}
We study infinite ternary words that contain few distinct palindromes.
In particular, we classify such words according to their critical exponents.
\end{abstract}

\renewcommand{\thefootnote}{\arabic{footnote}}
\setcounter{footnote}{0}

\section{Introduction}\label{sec:intro}

Repetitions have been a central theme in combinatorics on words since the pioneering work of~\cite{Thue06,Thue:1912}.  Here we study fractional repetitions of the form introduced by~\cite{Dejean1972}.  Roughly speaking, for a rational number $r>1$, a word $w$ is said to be an \emph{$r$-power} if it consists of a nonempty word $x$ repeated $r$ times.  For example, the English word $\texttt{alfalfa}$ is a $\frac{7}{3}$ power, and the French word $\texttt{chercher}$ is a $2$-power (also called a square).  For a real number $\beta\geq 1$, a word is said to be \emph{$\beta$-free} (resp.~\emph{$\beta^+$-free}) if it contains no $r$-powers with $r\geq \beta$ (resp.~$r>\beta$).  The \emph{critical exponent} of a word $w$ is the infimum of all $\beta$ such that $w$ is $\beta^+$-free.  For example, the word $\texttt{banana}$ has critical exponent $\tfrac{5}{2}$, since it is $\tfrac{5}{2}^+$-free, but contains the $\tfrac{5}{2}$-power $\texttt{anana}$.

\cite{Thue06} demonstrated the existence of an infinite square-free (i.e., $2$-free) ternary word, and \cite{Dejean1972} strengthened this result by showing that there is an infinite ternary word with critical exponent $\tfrac{7}{4}$, and that this is the minimum critical exponent among all infinite ternary words.  Roughly speaking, this says that the number $\tfrac{7}{4}$ represents the boundary between avoidable and unavoidable repetitions in infinite ternary words.

In general, the \emph{repetition threshold} for $k$ letters, denoted $\RT(k)$, is defined by
\[
\RT(k)=\inf\{\beta\in \mathbb{R}\colon\ \text{there is an infinite word over $k$ letters with critical exponent $\beta$}\}.
\]
In addition to establishing that $\RT(3)=\tfrac{7}{4}$, \cite{Dejean1972} conjectured that 
\[
\RT(k)=\begin{cases}
    \tfrac{7}{4}, &\text{ if $k=3$;}\\
    \tfrac{7}{5}, &\text{ if $k=4$;}\\
    \frac{k}{k-1}, &\text{ if $k=2$ or $k\geq 5$}.
\end{cases}
\]
This conjecture has since been proven through the work of many authors, and is now known as Dejean's theorem.  In particular, \cite{Carpi2007} proved all but finitely many cases, and the last cases were proven independently by~\cite{CurrieRampersad2011} and~\cite{Rao11}.

Especially since the completion of the proof of Dejean's theorem, much work has been done on determining the minimum critical exponent (i.e., the repetition threshold) among all infinite words in various restricted classes of words, including episturmian words (\cite{DvorakovaPelantova2024}), balanced words (\cite{DDP21,DvorakovaPelantovaOpocenskaShur2022,RampersadShallitVandomme2019}), rich words (\cite{CurrieMolPeltomaki2025,CurrieMolRampersad2020}), and words with given factor complexity (\cite{Currie2025,OllingerShallit2025,ShallitShur2019}).

Here we are concerned with determining the minimum critical exponent among all infinite words containing at most a fixed number of distinct palindromes.\footnote{For brevity, going forward every mention of a number of palindromes will refer to a number of distinct palindromes, including the empty word.}  In other words, we study the trade-off between the number of palindromes and the critical exponent in infinite words.
This trade-off has recently been investigated for binary words by~\cite{DOO2023}.
Removing the constraint on the alphabet size does not give anything new:
\begin{itemize}
\item The only infinite word containing at most $4$ palindromes is $(\texttt{012})^\omega$, which has infinite critical exponent.
\item For $k\geq 5$, it follows from Dejean's theorem that there is an infinite word (over $k-1$ letters) containing only $k$ palindromes with critical exponent $\RT(k-1)$, and this is best possible.
\end{itemize}
So in this paper, we consider this trade-off in infinite ternary words.
Our results completely answer questions of the following form:
Do infinite $\beta^+$-free ternary words with at most $p$ palindromes exist?
In every case that such words exist, we also determine whether there are exponentially or polynomially many such words.
The results are summarized in~\Cref{tab1}.
A white cell means that no such infinite word exists. Extremal white cells are labelled
with the corresponding item of Theorem~\ref{thm:neg}.
A green (resp.~red) cell means that there are exponentially (resp.~polynomially) many such words.
We have labelled the cells that correspond to an item of Theorem~\ref{thm:exp_pairs},~\ref{overlap}, or~\ref{41_22}.

\begin{table}[!htb]
\centerline{
\begin{tabular}{|l|l|l|l|l|l|l|l|l|}
\hline
18 & \cellcolor{green}\ref{thm:exp_pairs}\ref{e18} & \cellcolor{green} & \cellcolor{green} & \cellcolor{green} & \cellcolor{green} & \cellcolor{green} & \cellcolor{green} & \cellcolor{green}\\
\hline
17 & \ref{thm:neg}\ref{r16} & \cellcolor{red!50} & \cellcolor{green}\ref{thm:exp_pairs}\ref{e17} & \cellcolor{green} & \cellcolor{green} & \cellcolor{green} & \cellcolor{green} & \cellcolor{green}\\
\hline
16 &  & \cellcolor{red!50}\ref{41_22} &  \cellcolor{red!50} & \cellcolor{green}\ref{thm:exp_pairs}\ref{e16} & \cellcolor{green} & \cellcolor{green} & \cellcolor{green} & \cellcolor{green}\\
\hline
15 &  &  &  & \ref{thm:neg}\ref{r6} & \cellcolor{green} & \cellcolor{green} & \cellcolor{green} & \cellcolor{green}\\
\hline
14 &  &  &  &  & \cellcolor{green} & \cellcolor{green} & \cellcolor{green} & \cellcolor{green}\\
\hline
13 &  &  &  &  & \cellcolor{green} & \cellcolor{green} & \cellcolor{green} & \cellcolor{green}\\
\hline
12 &  &  &  &  & \cellcolor{green} & \cellcolor{green} & \cellcolor{green} & \cellcolor{green}\\
\hline
11 &  &  &  &  & \cellcolor{green} & \cellcolor{green} & \cellcolor{green} & \cellcolor{green}\\
\hline
10 &  &  &  &  & \cellcolor{green} & \cellcolor{green} & \cellcolor{green} & \cellcolor{green} \\
\hline
9 &  &  &  &  & \cellcolor{green} & \cellcolor{green} & \cellcolor{green} & \cellcolor{green} \\
\hline
8 &  &  &  &  & \cellcolor{green} & \cellcolor{green} &  \cellcolor{green} & \cellcolor{green} \\
\hline
7 &  &  &  &  & \cellcolor{green}\ref{thm:exp_pairs}\ref{e7} & \cellcolor{green} & \cellcolor{green} & \cellcolor{green} \\
\hline
6 &  &  &  &  & \cellcolor{red!50}\ref{overlap} & \cellcolor{green}\ref{thm:exp_pairs}\ref{e6} & \cellcolor{green} & \cellcolor{green} \\
\hline
5 &  &  &  &  &  & \ref{thm:neg}\ref{r5} & \cellcolor{green}\ref{thm:exp_pairs}\ref{e5} & \cellcolor{green} \\
\hline
4 &  &  &  &  &  &  & \ref{thm:neg}\ref{r4} & \cellcolor{red!50}$(\texttt{012})^\omega$ \\
\hline
\slashbox{$p$}{$\beta^+$} & $\tfrac74^+$ & $\tfrac{41}{22}^+$ & $\tfrac{25}{13}^+$ & $\tfrac{52}{27}^+$ & $2^+$ & $\tfrac94^+$ & $\tfrac{10}3^+$ & $\infty$ \\
\hline
\end{tabular}}
\caption{Infinite $\beta^+$-free ternary words with at most $p$ palindromes.}\label{tab1}
\end{table}

The remainder of the paper is organized as follows.  In Section~\ref{sec:pre}, we give the necessary background and definitions.  In Section~\ref{sec:neg}, we prove Theorem~\ref{thm:neg}, which establishes the negative results corresponding to the labelled white cells in~\Cref{tab1}.
In Section~\ref{sec:exp}, we prove Theorem~\ref{thm:exp_pairs}, which establishes the exponential growth in the cases corresponding to the labelled green cells in~\Cref{tab1}.
In Section~\ref{sec:poly1}, we prove Theorem~\ref{overlap}, which shows that there are only polynomially many overlap-free words (i.e., $2^+$-free words) with at most 6 palindromes.
The remaining polynomial cases are covered in Section~\ref{sec:poly2} and Section~\ref{sec:ce}.  In Section~\ref{sec:poly2}, we relate the structure of all infinite $\tfrac{52}{27}$-free words with at most $16$ palindromes and all $\tfrac{25}{13}$-free words with at most $17$ palindromes to one particular infinite word with $16$ palindromes.  The proof that this word has critical exponent $\tfrac{41}{22}$ is long and technical, and is given in Section~\ref{sec:ce}.
Finally, in Section~\ref{sec:related}, we observe that for ternary square-free words, having at most most 16 palindromes is equivalent to avoiding overpals and also to avoiding the letter pattern $abcacba$.  We use this observation to prove a conjecture of~\cite{ARS2017} on words avoiding overpals.  We also simplify the proofs (and improve certain results) of~\cite{P2016}
on the minimum critical exponent of square-free ternary words avoiding other letter patterns.

Several proofs throughout the paper make use of computer checks.  The programs used can be found at the link below.
\[
\href{http://www.lirmm.fr/~ochem/morphisms/palin3.htm}{\texttt{http://www.lirmm.fr/~ochem/morphisms/palin3.htm}}
\]
Throughout the paper, to check that a given word does not contain more than the claimed number of palindromes, observe that we only need to check the factors up to length $k+2$, where $k$ is the length of the longest allowed palindrome.

\section{Preliminaries}\label{sec:pre}
An \emph{alphabet} $\mathcal A$ is a finite set and its elements are called \emph{letters}. 
A \emph{word} $u$ over $\mathcal A$ of \emph{length} $n$ is a finite string $u = u_0 u_1 \cdots u_{n-1}$, where $u_j\in\mathcal A$ for all $j \in \{0,1,\dots, n-1\}$.
The length of $u$ is denoted $|u|$ and $|u|_{f}$ denotes the number of occurrences of the factor ${f}\in\mathcal A^+$ in the word $u$. If ${\mathcal A}=\{\tt 0, \tt 1, \dots, \tt d-1\}$, the \emph{Parikh vector} $ \vec{u} \in \mathbb N^{d}$ is the vector defined as ${\vec u} = (|u|_{\tt 0}, |u|_{\tt 1}, \dots, |u|_{\tt d-1})^{ T}$.
The set of all finite words over $\mathcal A$ is denoted ${\mathcal A}^*$. The set ${\mathcal A}^*$ equipped with concatenation as the operation forms a monoid with the \emph{empty word} $\varepsilon$ as the neutral element. If $w=ux$ for some $u,x \in {\mathcal A}^*$, then $wx^{-1}=u$.
We will also consider the set ${\mathcal A}^\omega$ of infinite words (that is, right-infinite words) and the set ${}^\omega\hspace{-0,15cm}{\mathcal A}^\omega$ of bi-infinite words.
A word $v$ is an $e$-\emph{power} of a word $u$ if $v$ is a~prefix of the infinite periodic word $uuu\cdots = u^\omega$ and $e=|v|/|u|$.
We write $v=u^e$. We also call $u^e$ a repetition with period $u$ and exponent $e$.
For instance, the Czech word $\texttt{kavka}$ (a kind of bird -- jackdaw) can be written in this formalism as $(\texttt{kav})^{5/3}$.
A word (finite or infinite) is \emph{$\alpha^+$-free} (resp. \emph{$\alpha$-free}) if it 
contains no repetition with exponent $\beta$ such that $\beta>\alpha$ (resp. $\beta\ge\alpha$).
For example, the word $\texttt{kavka}$ is $\frac{5}{3}^+$-free and square-free (i.e., $2$-free), but not $\frac{5}{3}$-free.

The \emph{critical exponent}  $E({\bf u} )$ of an infinite word ${\bf u}$ is defined as
$$
E({\bf u}) =\inf\acc{\beta\in\mathbb{R}: {\bf u} \text{ is $\beta^+$-free}}
$$
or equivalently as
$$E({\bf u}) =\sup\acc{e \in \mathbb{Q}: \ v ^e \text{ is a factor of } {\bf u} \text{ for a non-empty word } v}.$$
If each factor of ${\bf u}$ has infinitely many occurrences in ${\bf u}$, then ${\bf u}$ is \emph{recurrent}.
Moreover, if for each factor the distances between its consecutive occurrences are bounded, then ${\bf u}$ is \emph{uniformly recurrent}. 
The \emph{reversal} of the word $w=w_0w_1\cdots w_{n-1}$, where the $w_i$ are letters, is the word $w^R=w_{n-1}\cdots w_1 w_0$.
A word $w$ is a~\emph{palindrome} if $w=w^R$.


Consider a factor $w$ of a recurrent infinite word ${\bf u} = u_0 u_1 u_2 \cdots$. Let $j < \ell$ be two consecutive occurrences of $w$ in $\bf u$, i.e., $u_j u_{j+1} \cdots u_{j+|w|-1}=w=u_{\ell} u_{\ell+1} \cdots u_{\ell+|w|-1}$ and $u_ku_{k+1}\cdots u_{k+|w|-1}\neq w$ for all $k$ satisfying $j<k<\ell$. Then the word $u_j u_{j+1} \cdots u_{\ell-1}$ is a~\emph{return word} to $w$ in~$\bf u$.


Let $\LL({\bf u})$ denote the set of factors of $\bf u$, sometimes referred to as the \emph{language} of $\bf{u}$.
Given a word $w \in \LL({\bf u})$, we define the sets of left extensions, right extensions, and bi-extensions of $w$ in ${\bf u}$ over an alphabet $\mathcal A$ respectively as
$$
\mathrm{Lext}_{{\bf u}}(w) = \{ a \in {\mathcal A} : aw \in \LL({\bf u}) \},
\qquad
\mathrm{Rext}_{{\bf u}}(w) = \{ b \in {\mathcal A} : wb \in \LL({\bf u}) \}
$$
and
$$
\mathrm{Bext}_{{\bf u}}(w) = \{ (a,b) \in {\mathcal A} \times {\mathcal A} : awb \in \LL({\bf u}) \}. $$

If $\#\mathrm{Lext}_{{\bf u}}(w)>1$, then $w$ is called \emph{left special (LS)}. If $\#\mathrm{Rext}_{{\bf u}}(w)>1$, then $w$ is called \emph{right special (RS)}. If $w$ is both LS and RS, then it is called \emph{bispecial (BS)}.
The \emph{bilateral order~\footnote{The bilateral order was introduced by~\cite{Cassaigne1997} as a tool for the computation of factor complexity.}}
 of a factor $w$ is defined as $b(w)=\#\mathrm{Bext}_{{\bf u}}(w)-\#\mathrm{Lext}_{{\bf u}}(w)-\#\mathrm{Rext}_{{\bf u}}(w)+1$ and we distinguish \emph{ordinary BS factors} with $b(w)=0$, \emph{weak BS factors} with $b(w)<0$, and \emph{strong BS factors} with $b(w)>0$.

A \emph{morphism} is a map $\psi: {\mathcal A}^* \to {\mathcal B}^*$ such that $\psi(uv) = \psi(u)\psi(v)$  for all words $u, v \in {\mathcal A}^*$.
The morphism $\psi$ is \emph{non-erasing} if $\psi(a)\not =\varepsilon$ for each $a \in {\mathcal A}$.
Morphisms can be naturally extended to infinite words by setting
$\psi(u_0 u_1 u_2 \cdots) = \psi(u_0) \psi(u_1) \psi(u_2) \cdots\,$.
A \emph{fixed point} of a morphism $\psi:  {\mathcal A}^* \to  {\mathcal A}^*$ is an infinite word $\bf u$ such that $\psi(\bf u) = \bf u$.

If $\mathcal{A}=\{\tt{0},\tt{1},\ldots, \tt{d-1}\}$, then for each morphism $\psi: {\mathcal A}^* \to  {\mathcal A}^*$ there is an associated \emph{(incidence) matrix} $M_\psi$ defined for all $i,j \in \{0,1,\dots, d-1\}$ by $[M_\psi]_{ij}=|\psi(\tt j)|_{\tt i}$.   Note that multiplying the Parikh vector of a word $u \in {\mathcal A}^*$ by the incidence matrix $M_\psi$ gives the Parikh vector of $\psi(u)$, i.e., we have $\vv{\psi(u)}=M_\psi\vec{u}$ for all $u\in\mathcal{A}^*$.
If there exists $N\in \mathbb N$ such that $M_\psi^N$ has all positive entries, then $\psi$ is called a~\emph{primitive} morphism.


Let ${\bf u}$ be an infinite word over an alphabet $\mathcal A$ and let $\psi:{\mathcal A}^* \to {\mathcal B}^*$ be a~morphism. 
Consider a factor $w$ of $\psi({\bf u})$. We say that $(w_1, w_2)$ is a \emph{synchronization point} of $w$ if $w=w_1w_2$ and for all $p,s \in\LL(\psi({\bf u}))$ and $v \in\LL({\bf u})$ such that $\psi(v)=pws$ there exists a factorization $v=v_1v_2$ of $v$ with $\psi(v_1)=pw_1$ and $\psi(v_2)=w_2s$. We denote the synchronization point by $w_1\bullet w_2$. For instance, consider the morphism $\psi:\Sigma_2^*\rightarrow \Sigma_3^*$ defined by $\psi(\texttt{0})=\texttt{012}$ and  $\psi(\texttt{1})=\texttt{0012}$. Then the factor $\tt 12001$ of $\psi(({\tt 01})^\omega)$ has the synchronization point ${\tt 12}\bullet\tt{001}$ since $\tt 2$ occurs only as the last letter of $\psi({\tt 0})$ and $\psi({\tt 1})$.

Given a factorial language $L$ (a language where every factor of a word in the language is also part of that language) and an integer $\ell$, let $L^\ell$ denote the words of length $\ell$ in $L$.
The \emph{Rauzy graph} of $L$ of order $\ell$ is the directed graph
whose vertices are the words of $L^{\ell-1}$, the arcs are the words of $L^{\ell}$, and the arc
corresponding to the word $w$ goes from the vertex corresponding to the prefix of $w$ of length $\ell -1$
to the vertex corresponding to the suffix of $w$ of length $\ell -1$.


An \emph{overpal} is a word of the form $axax^Ra$, where $x^R$ is
the reversal of the (possibly empty) word $x$ and $a$ is a single letter.
Overpals were introduced by~\cite{ARS2017}.

A \emph{pattern} $p$ is a non-empty finite word over the alphabet
$\Delta=\acc{A,B,C,\dots}$ of capital letters called \emph{variables}.
An \emph{occurrence} of $p$ in a word $w$ is a non-erasing morphism $h:\Delta^*\to\Sigma^*$
such that $h(p)$ is a factor of $w$.
A \emph{letter pattern} is a non-empty finite word over the alphabet $\acc{a, b, c, \ldots}$ of lowercase variables
such that every variable stands for one letter from the alphabet $\acc{\texttt{0}, \texttt{1}, \texttt{2}, \ldots}$
and different variables denote different letters.
For example, the occurrences of the letter pattern $abaca$ over the ternary alphabet are
\texttt{01020}, \texttt{02010}, \texttt{10121}, \texttt{12101}, \texttt{20212}, \texttt{21202}.
Letter patterns were introduced by~\cite{P2016}.  We say that the word $w$ \emph{avoids} the pattern (or letter pattern) $p$ if $w$ contains no occurrence of $p$.

\section{Negative results}\label{sec:neg}


\begin{theorem}\label{thm:neg}
There exists no infinite ternary $\beta$-free word containing only $p$ palindromes
for the following pairs $(p, \beta)$.
\begin{enumerate}[label=\normalfont(\alph*)]
\item $(4,\beta)$\label{r4} for any fixed $\beta$.
\item $(5,\tfrac{10}3)$\label{r5}
\item $(15,2)$\label{r6}
\item $(17,\tfrac{41}{22})$\label{r16}
\end{enumerate}
\end{theorem}

\begin{proof}
Item~\ref{r4} follows from the fact that, up to renaming of the letters, the word $(\texttt{012})^\omega$ is the only infinite word with at most 4 palindromes.  The other cases are obtained by computer backtracking searches, the results of which are described below.
\begin{itemize}
    \item The longest $10/3$-free ternary word with at most $5$ palindromes has length $70$.
    \item The longest $2$-free ternary word with at most $15$ palindromes has length $42$.
    \item The longest $41/22$-free ternary word with at most $17$ palindromes has length $449$. \qedhere
\end{itemize}
\end{proof}

\section{Exponential cases}\label{sec:exp}
We need some terminology and a lemma by~\cite{Mol&Rampersad&Shallit:2020}.
A morphism $f:\Sigma^*\rightarrow\Delta^*$ is \emph{$q$-uniform\/} if $|f(a)|=q$ for every $a\in\Sigma$, and is called \emph{synchronizing} if for all $a,b,c\in\Sigma$ and $u,v\in \Delta^*$, if $f(ab)=uf(c)v$, then either $u=\varepsilon$ and $a=c$, or $v=\varepsilon$ and $b=c$.
\begin{lemma}[\cite{Mol&Rampersad&Shallit:2020}, Lemma 23]\label{lem:MRS}
Let $r,s\in\mathbb{R}$ satisfy $1<r<s$. Let $\rho\in\{r,r^+\}$ and $\sigma\in\{s,s^+\}$.
Let $h\colon \Sigma^*\rightarrow \Delta^*$ be a synchronizing $q$-uniform morphism. Set
$$t = \max\paren{\frac{2s}{s-r},\frac{2(q-1)(2s-1)}{q(s-1)}}.$$
If $h(w)$ is $\sigma$-free for every $\rho$-free word $w\in\Sigma^*$ with
$|w|\le t$, then $h(z)$ is $\sigma$-free for every $\rho$-free word $z \in \Sigma^*$.
\label{mrs}
\end{lemma}
Most results in this subsection use the following steps.
We find an appropriate uniform synchronizing morphism $h$ by exhaustive search.
We use Lemma~\ref{lem:MRS} to show that $h$ maps every ternary $\tfrac74^+$-free word
(resp. 4-ary $\tfrac75^+$-free word) to a suitable ternary $\beta^+$-free word.
Since there are exponentially many ternary $\tfrac74^+$-free words
(resp. 4-ary $\tfrac75^+$-free words), ref.~\cite{Ochem2004}, there are also exponentially many ternary $\beta^+$-free words with the desired number of palindromes.

\begin{theorem}\label{thm:exp_pairs}
There exist exponentially many infinite ternary $\beta^+$-free words containing at most $p$ palindromes
for the following pairs $(p, \beta)$.
\begin{enumerate}[label=\normalfont(\alph*)]
\item $(5,\tfrac{10}3)$\label{e5}
\item $(6,\tfrac94)$\label{e6}
\item $(7,2)$\label{e7}
\item $(16,\tfrac{52}{27})$\label{e16}
\item $(17,\tfrac{25}{13})$\label{e17}
\item $(18,\tfrac74)$\label{e18}
\end{enumerate}
\end{theorem}

\begin{proof}{\ }
\begin{enumerate}[label=(\alph*)]
\item $(5,\tfrac{10}3)$: We claim that applying the following morphism $f$ to any binary cube-free word gives a ternary $\tfrac{10}3^+$-free word containing only $5$ palindromes.
\begin{align*}
\texttt{0}&\rightarrow\texttt{012} \\
\texttt{1}&\rightarrow\texttt{0012}
\end{align*}
Since there are exponentially many cube-free binary words, ref.~\cite{Brandenburg1983}, it follows that there are exponentially many $\tfrac{10}{3}^+$-free ternary words containing at most $5$ palindromes. 

To prove the claim, let $w$ be a binary cube-free word. Suppose towards a contradiction that $f(w)$ contains a $\tfrac{10}{3}^+$-power, say $X=U^Q$, where $Q>\tfrac{10}{3}$. First suppose that $|U|\le 17$. Then $X$ contains a $\tfrac{10}{3}^+$-power of length at most $4|U|\le 68$. But every factor of $f(w)$ of length at most $68$ appears in the $f$-image of some cube-free word $y$ of length at most $\tfrac{68-2}{3}+2=24$. Using a computer, we enumerate every binary cube-free word $y$ of length at most $24$ and check that $f(y)$ is $\tfrac{10}{3}^+$-free. So we may assume that $|U|\ge 18$. Then we can write $X=\alpha f(x)\beta$, where $\alpha$ is a proper suffix of $f(\tt{0})$ or $f(\tt{1})$, $\beta$ is a proper prefix of $f(\tt{0})$ or $f(\tt{1})$, and $x$ is a factor of $w$. Note that $|\alpha|\le 3$ and $|\beta|\le 3$. 
Since $|U|\ge 18$ and $X=U^Q$ for some $Q>\tfrac{10}{3}$, we have 
\[
|f(x)|= |X|-|\alpha|-|\beta|>\tfrac{10}{3}|U|-6\ge 3|U|.
\]
So $f(x)$ has exponent at least $3$. It follows that $x$ has exponent at least $3$, which contradicts the assumption that $w$ is cube-free.

\item $(6,\tfrac94)$: Applying the following 25-uniform morphism to any ternary $\tfrac74^+$-free word gives a ternary $\tfrac94^+$-free word containing only $6$ palindromes.
\begin{align*}
\texttt{0}&\rightarrow\texttt{0112001200112011200112012} \\
\texttt{1}&\rightarrow\texttt{0012011200112012001120012} \\
\texttt{2}&\rightarrow\texttt{0011201120012011201200112}
\end{align*}

\item $(7,2)$: Applying the following 4-uniform morphism to any ternary $\tfrac74^+$-free word gives a ternary $2^+$-free word containing only $7$ palindromes.
\begin{align*}
\texttt{0}&\rightarrow\texttt{0012} \\
\texttt{1}&\rightarrow\texttt{0112} \\
\texttt{2}&\rightarrow\texttt{0122}
\end{align*}

\item $(16,\tfrac{52}{27})$: Applying the following 609-uniform morphism to any ternary $\tfrac74^+$-free word gives a ternary $\tfrac{52}{27}^+$-free word containing only $16$ palindromes.
{\footnotesize
\begin{align*}
\texttt{0}\rightarrow&\ p\texttt{21012010201210120212010201202120121012010201202120102012101202120121012010201}\\ 
&\ \texttt{210120212010201202120121012021201020121012010201202120102012101202120121012010}\\
&\ \texttt{201210120212010201202120121012010201202120102012101202120121012010201202120121}\\
&\ \texttt{012021201020121012010201202120121012010201210120212010201202120121012021201020}\\
&\ \texttt{121012010201202120102012101202120121012010201210120212010201202120121012021201}\\
&\ \texttt{0201210120102} \\
\texttt{1}\rightarrow&\ p\texttt{02012101202120121012010201202120121012021201020121012010201202120121012010201}\\
&\ \texttt{210120212010201202120121012021201020121012010201202120102012101202120121012010}\\
&\ \texttt{201210120212010201202120121012021201020121012010201202120121012010201210120212}\\
&\ \texttt{010201202120121012010201202120102012101202120121012010201202120121012021201020}\\
&\ \texttt{121012010201202120121012010201210120212010201202120121012021201020121012010201}\\
&\ \texttt{2021201020121} \\
\texttt{2}\rightarrow&\ p\texttt{02012101202120121012010201202120121012021201020121012010201202120121012010201}\\
&\ \texttt{210120212010201202120121012021201020121012010201202120102012101202120121012010}\\
&\ \texttt{201210120212010201202120121012010201202120102012101202120121012010201202120121}\\
&\ \texttt{012021201020121012010201202120121012010201210120212010201202120121012010201202}\\
&\ \texttt{120102012101202120121012010201210120212010201202120121012021201020121012010201}\\
&\ \texttt{2021201020121}
\end{align*}}
where {\footnotesize $p=\texttt{01202120121012010201210120212010201202120121012010201202120102012101202}\\\texttt{120121012010201202120121012021201020121012010201202120102012101202120121012010201}\\\texttt{2101202120102012021201210120212010201210120102012021201}$.}
\item $(17,\tfrac{25}{13})$: Applying the following 121-uniform morphism to any ternary $\tfrac74^+$-free word gives a ternary $\tfrac{25}{13}^+$-free word containing only $17$ palindromes.
{\footnotesize
\begin{align*}
\texttt{0}\rightarrow p\texttt{1202101201021202102010210120210201021202101201020120210120102120210}\\
\texttt{1}\rightarrow p\texttt{1202101201020120210120102120210201021012021020102120210120102120210}\\
\texttt{2}\rightarrow p\texttt{1012021201020120210120102120210201021012021020102120210120102012021}
\end{align*}
}
where {\footnotesize$p=\texttt{201021012021201020120210120102012021201021012021020102}$}.
\item $(18,\tfrac74)$: Applying the following 87-uniform morphism to any 4-ary $\tfrac75^+$-free word gives a ternary $\tfrac74^+$-free word containing only $18$ palindromes.
{\footnotesize
\begin{align*}
\texttt{0}\rightarrow p\texttt{2101202102010212021012010201210120210201021012021201210}\\
\texttt{1}\rightarrow p\texttt{2101202102010210121021202101201020121012021020102120210}\\
\texttt{2}\rightarrow p\texttt{2101202102010210120212010201210120210201021012102120210}\\
\texttt{3}\rightarrow p\texttt{0201202101210212021012010201202120121012021020102101202}
\end{align*}}
where {\footnotesize$p=\texttt{12010201202101210201021012021201}$}.\qedhere
\end{enumerate}
\end{proof}

Trying to estimate the growth rate in the cases above is certainly a difficult task.
However, we note that~\cite{FleischerShallit} have shown that the number of ternary words
of length $n$ with at most 5 palindromes (sequence \href{https://oeis.org/A329023}{A329023} in the OEIS) is $\Theta\paren{\kappa^n}$, where $\kappa=1.2207440846\cdots$ is a root of $x^4=x+1$.

\section{First polynomial case}\label{sec:poly1}
We consider the morphic word $t^\omega(\texttt{0})$, where $t$ is the morphism defined as follows.
\begin{align*}
t(\texttt{0})&=\texttt{01120} \\
t(\texttt{1})&=\texttt{12001} \\
t(\texttt{2})&=\texttt{2} \\
\end{align*}

\begin{lemma}\label{Lem:thetaTM}
The word $t^\omega(\normaltt{0})$ is obtained from the Thue-Morse word
(i.e., the fixed point $\mu^\omega(\tt{0})$ of the morphism $\mu$ defined by $\normaltt{0}\mapsto\normaltt{01}$ and $\normaltt{1}\mapsto\normaltt{10}$)
by inserting the letter {\normaltt{2}} in the middle of every factor {\normaltt{10}}.
\end{lemma}

\begin{proof}
If we erase every occurrence of \texttt{2} from $t^\omega(\texttt{0})$,
we get the fixed point of $\texttt{0}\mapsto\texttt{0110}$; $\texttt{1}\mapsto\texttt{1001}$,
that is, the Thue-Morse word. Now let us check that the letter \texttt{2}
is inserted correctly. The word $t^\omega(\texttt{0})$ avoids \texttt{10}, so that \texttt{2}
must have been inserted in the middle of every factor $\texttt{10}$.
The fixed point of $t$ also avoids \texttt{02}, \texttt{21}, and $\texttt{22}$, so that \texttt{2}
is not inserted anywhere else.
\end{proof}

\begin{theorem}\label{overlap}
The word $t^\omega(\normaltt{0})$ contains 6 palindromes and is $2^+$-free.
Every bi-infinite ternary $\tfrac94$-free word containing at most 6 palindromes has the same factor set as one
of the six words obtained from $t^\omega(\normaltt{0})$ by letter permutation.
\end{theorem}
\begin{proof}
Rosenfeld has written a program that implements the algorithm of~\cite{cassaignealgo} to test whether a morphic word $w$ avoids a pattern $P$, with $w$ and $P$ as input.
We use it to show that $t^\omega(\texttt{0})$ avoids the pattern $ABABA$. Moreover, $t^\omega(\texttt{0})$ avoids \texttt{000},
\texttt{111}, and \texttt{22}. So, $t^\omega(\texttt{0})$ is $2^+$-free.
Alternatively, one could argue that the insertion of the letter \texttt{2}
in the middle of every factor $\texttt{10}$ in the Thue-Morse word does not create an overlap.

Suppose now that $\mathbf{w}$ is a ternary $\tfrac{9}{4}$-free word containing at most $6$ palindromes.
Note that $\mathbf{w}$ must contain the four palindromes $\varepsilon$, $\texttt{0}$, $\texttt{1}$, and $\texttt{2}$.
Further, we see that $\mathbf{w}$ must contain two palindromes from the set $\{\texttt{00},\texttt{11},\texttt{22}\}$, because a backtracking search shows that there are only finitely many ternary $\tfrac{9}{4}$-free words containing at most six palindromes and containing at most one palindrome from $\{\texttt{00},\texttt{11},\texttt{22}\}$.
By permuting the letters if necessary, we assume that $\mathbf{w}$ contains the palindromes $\texttt{00}$ and $\texttt{11}$.

So $\mathbf{w}$ contains the six palindromes $\varepsilon$, $\texttt{0}$, $\texttt{1}$, $\texttt{2}$, $\texttt{00}$, and $\texttt{11}$, and hence no $\texttt{22}$ and no palindrome of the form $aba$ or $abba$, where $a$ and $b$ are letters. Thus, it is not hard to see that every return word to $\tt{2}$ in $\mathbf{w}$ belongs to $R_1\cup R_2$, where
\[
R_1=\{\texttt{201},\texttt{2001},\texttt{2011},\texttt{20011}\}
\]
and 
\[
R_2=\{\texttt{210},\texttt{2110},\texttt{2100},\texttt{21100}\}.
\]
Further, we see that $\mathbf{w}$ does not contain factors from both $R_1$ and $R_2$ (otherwise it would contain either $\tt{020}$ or $\tt{121}$). By permuting $\texttt{0}$ and $\texttt{1}$ if necessary, we may assume that $\mathbf{w}$ contains no factors from $R_2$. It follows that $\mathbf{w}$ is obtained from a binary word $\mathbf{v}$ by inserting the letter $\texttt{2}$ in the middle of every factor $\texttt{10}$. (Note that $\mathbf{v}$ is obtained from $\mathbf{w}$ by deleting every $\texttt{2}$.)

We claim that $\mathbf{v}$ is $\tfrac{7}{3}$-free. Suppose otherwise that $\mathbf{v}$ contains a factor, say $u$, of exponent greater than $\tfrac{7}{3}$. It follows that $u$ has a prefix of the form $xyxyx$, where $|xyxyx|/|xy|>\tfrac{7}{3}$, or equivalently 
\begin{align}\label{Ineq:7/3}
    |y|<2|x|.
\end{align}  
Let $U$ be the word obtained from $xyxyx$ by inserting the letter $\texttt{2}$ in the middle of every factor $\texttt{10}$, so that $U$ is a factor of $\mathbf{w}$. Let $X$ be the word obtained from $x$ by inserting the letter $\texttt{2}$ in the middle of every factor $\tt{10}$, and let $Y$ be the word such that $V=XYXYX$. (So $Y$ is obtained from $y$ by inserting the letter $\texttt{2}$ in the middle of every factor $\texttt{10}$, and possibly adding a $\texttt{2}$ at the beginning and/or the end.)

If $|x|\le 6$, then we have $|y|<2|x|\le 12$, hence there are only finitely many possibilities for the word $xyxyx$. For each possible word $xyxyx$, we check by computer that the corresponding word $U=XYXYX$ has exponent greater than $\tfrac{9}{4}$, which contradicts the assumption that $\mathbf{w}$ is $\tfrac{9}{4}$-free. So we may assume that $|x|\ge 7$. 

We now claim that 
\begin{align}\label{Ineq:X}
    |X|\ge \tfrac{5}{4}|x|-1
\end{align} 
and 
\begin{align}\label{Ineq:Y}
    |Y|\le \tfrac{3}{2}|y|+2.
\end{align}
For (\ref{Ineq:X}), note that $|X|=|x|+|x|_{\texttt{10}}$. Since $\mathbf{w}$ (and $\mathbf{v}$, in turn) contains neither $\tt{000}$ nor $\tt{111}$, the factor $\tt{10}$ must occur at least $\tfrac{|x|}{4}-1$ times in $x$ if $|x|$ is a multiple of $4$, and at least $\floor{\tfrac{|x|}{4}}$ times otherwise. Thus we have
\[
|X|=|x|+|x|_{\texttt{10}}\ge |x|+\tfrac{|x|}{4}-1=\tfrac{5}{4}|x|-1,
\]
as desired. For (\ref{Ineq:Y}), we have
\[
|Y|\le |y|+|y|_{\texttt{10}}+2\le |y|+\tfrac{|y|}{2}+2=\tfrac{3}{2}|y|+2
\]
by similar reasoning.
Thus, we have
\begin{align*}
    |Y|&\le \tfrac{3}{2}|y|+2 & & \text{\footnotesize by (\ref{Ineq:Y})}\\
    &<3|x|+2  & & \text{\footnotesize by (\ref{Ineq:7/3})}\\
    &\le \tfrac{12}{5}|X|+\tfrac{22}{5} & & \text{\footnotesize by (\ref{Ineq:X})}\\
    &<3|X| & & \text{\footnotesize since $|x|\ge 7$, hence $|X|\ge 8$.}
\end{align*}
Thus, we have shown that $|Y|<3|X|$, which is equivalent to $|XYXYX|/|XY|>\tfrac{9}{4}$. In other words, we have shown that $U=XYXYX$ has exponent greater than $\tfrac{9}{4}$, which contradicts the assumption that $\mathbf{w}$ is $\tfrac{9}{4}$-free. 

Therefore, we conclude that $\mathbf{v}$ is $\tfrac{7}{3}$-free. Now it follows from a well-known theorem of~\cite{KarhumakiShallit2004} that $\mathbf{v}$ has the same factor set as the Thue-Morse word. Thus, by Lemma~\ref{Lem:thetaTM}, we conclude that $\mathbf{w}$ has the same factor set as $t^\omega(\texttt{0})$.
\end{proof}

\section{Second polynomial case}\label{sec:poly2}
We consider the morphic word $\gamma(\eta^\omega(\texttt{0}))$ defined by the following morphisms.
\begin{align*}
    \eta(\texttt{0}) &= \texttt{010203} & \gamma(\texttt{0}) &= \texttt{012021201210120102}\\
    \eta(\texttt{1}) &= \texttt{2013} & \gamma(\texttt{1}) &= \texttt{012101202120102}\\
    \eta(\texttt{2}) &= \texttt{0132} & \gamma(\texttt{2}) &= \texttt{012021201020121}\\
    \eta(\texttt{3}) &= \texttt{0102013203} & \gamma(\texttt{3}) &= \texttt{012021201210120212010201210120102}
\end{align*}
We say that a $4$-ary word (finite or infinite) is $\eta$-good if it avoids $AA$ and
\begin{align*}
F_\eta&=\{\texttt{12},\texttt{21},\texttt{23},\texttt{31},\texttt{103},\texttt{302},\texttt{303},\texttt{132013},\texttt{320132},\\
& \ \ \ \ \ \ \ \ \texttt{2010201},\texttt{2013201},\texttt{3013203},\texttt{030102030}\}.
\end{align*}
We say that a ternary word (finite or infinite) is $\gamma$-good if it avoids $AA$ and
\begin{align*}
F_\gamma&=\{\texttt{0210},\texttt{1021},\texttt{2102},\texttt{012010201210120212012},\texttt{020121012021201020121},\\ & \ \ \ \ \ \ \ \ \texttt{101202120102012101202},\texttt{120121012021201020120},\\ & \ \ \ \ \ \ \ \ \texttt{201202120102012101201},\texttt{212010201210120212010}\}.
\end{align*}

\begin{lemma}\label{hgood}
A bi-infinite word is $\eta$-good if and only if it has the same factor set as $\eta^\omega(\normaltt{0})$.
\end{lemma}
\begin{proof}
It is easy to check that $\eta^\omega(\texttt{0})$ avoids $F_\eta$ and we use Rosenfeld's program to show that $\eta^\omega(\texttt{0})$ avoids $AA$.
So $\eta^\omega(\texttt{0})$ is $\eta$-good.

We construct the set $S_\eta^{58}$ defined as follows:
a word $v$ is in $S_\eta^{58}$ if and only if there exists an $\eta$-good word $pvs$ such that $|p|=|v|=|s|=58$.
We check that $S_\eta^{58}$ is exactly the set of factors of length 58 of $\eta^\omega(\texttt{0})$.
Thus, if $w$ is any bi-infinite $\eta$-good word, then its factors of length 58 are also factors of $\eta^\omega(\texttt{0})$.
Now, every element of $S_\eta^{58}$ contains the factor $\eta(\texttt{0})=\texttt{010203}$
and every element of $S_\eta^{58}$ with prefix $\eta(\texttt{0})$ has a prefix in
\[ \acc{\eta(\texttt{010}),\eta(\texttt{020}),\eta(\texttt{030}),\eta(\texttt{0130}),\eta(\texttt{0320}),\eta(\texttt{01320})}.
\]
So $w\in\{\eta(\texttt{01}),\eta(\texttt{02}),\eta(\texttt{03}),\eta(\texttt{013}),\eta(\texttt{032}),\eta(\texttt{0132})\}^\omega
$, hence $w\in\acc{\eta(\texttt{0}),\eta(\texttt{1}),\eta(\texttt{2}),\eta(\texttt{3})}^\omega$.
Thus we have $w=\eta(u)$ for some bi-infinite 4-ary word $u$.
Notice that $u$ must be square-free since $w=\eta(u)$ is square-free.
We also check that for every factor $f\in F_\eta$, the set $S_\eta^{58}$ does not contain $\eta(f)$.
Since $\max_{f\in F_\eta}|\eta(f)|=58$, the pre-image $u$ of $w$ does not contain any factor in $F_\eta$.
So $u$ is also $\eta$-good. By induction, $u$ and $w$ have the same factor set as $\eta^\omega(\texttt{0})$.
\end{proof}

\begin{lemma}\label{ggood}
A bi-infinite word is $\gamma$-good if and only if it has the same factor set as $\gamma(\eta^\omega(\normaltt{0}))$.
\end{lemma}
\begin{proof}
It is easy to check that $\gamma(\eta^\omega(\texttt{0}))$ avoids $F_\gamma$ and we use Rosenfeld's program to show that $\gamma(\eta^\omega(\texttt{0}))$ avoids $AA$.
So $\gamma(\eta^\omega(\texttt{0}))$ is $\gamma$-good.

We construct the set $S_\gamma^{186}$ defined as follows:
a word $v$ is in $S_\gamma^{186}$ if and only if there exists a $\gamma$-good word $pvs$ such that $|p|=|v|=|s|=186$.
We check that $S_\gamma^{186}$ is exactly the set of factors of length 186 of $\gamma(\eta^\omega(\texttt{0}))$.
Thus, if $w$ is any bi-infinite $\gamma$-good word, then its factors of length 186 are also factors of $\gamma(\eta^\omega(\texttt{0}))$.
Now, every element of $S_\gamma^{186}$ contains the factor $\gamma(\texttt{0})=\texttt{012021201210120102}$
and every element of $S_\gamma^{186}$ with prefix $\gamma(\texttt{0})$ has a prefix in
\[
\acc{\gamma(\texttt{010}),\gamma(\texttt{020}),\gamma(\texttt{030}),\gamma(\texttt{0130}),\gamma(\texttt{0320}),\gamma(\texttt{01320})}.
\]
So $w\in\{\gamma(\texttt{01}),\gamma(\texttt{02}),\gamma(\texttt{03}),\gamma(\texttt{013}),\gamma(\texttt{032}),\gamma(\texttt{0132})\}^\omega$, hence $w\in\acc{\gamma(\texttt{0}),\gamma(\texttt{1}),\gamma(\texttt{2}),\gamma(\texttt{3})}^\omega$.
Thus we have $w=\gamma(u)$ for some bi-infinite 4-ary word~$u$.
Notice that $u$ must be square-free since $w=\gamma(u)$ is square-free.
We also check that for every factor $f\in F_\eta$, the set $S_\gamma^{186}$ does not contain $\gamma(f)$.
Since $\max_{f\in F_\eta}|\gamma(f)|=186$, the pre-image $u$ of $w$ does not contain any factor in $F_\eta$.
So $u$ is $\eta$-good. By \Cref{hgood}, $u$ has the same factor set as $\eta^\omega(\texttt{0})$.
So $w$ has the same factor set as $\gamma(\eta^\omega(\texttt{0}))$.
\end{proof}


\begin{theorem}\label{41_22}{\ }
\begin{enumerate}[label=\normalfont(\alph*)]
\item $\gamma(\eta^\omega(\normaltt{0}))$ contains exactly 16 palindromes and is $\tfrac{41}{22}^+$-free.\label{it:16p}
\item Every bi-infinite ternary $\tfrac{52}{27}$-free word containing at most $16$ palindromes has the same factor set as either $\gamma(\eta^\omega(\normaltt{0}))$ or $\gamma(\eta^\omega(\normaltt{0}))^R$.\label{it:pos}
\item Every recurrent ternary $\tfrac{25}{13}$-free word containing at most $17$ palindromes has the same factor set as either $\gamma(\eta^\omega(\normaltt{0}))$ or $\gamma(\eta^\omega(\normaltt{0}))^R$.\label{it:neg}
\end{enumerate}
\end{theorem}
\begin{proof}
Let us prove \Cref{it:16p}.  The word
$\gamma(\eta^\omega(\texttt{0}))$ contains the palindromes $\varepsilon$, \texttt{0}, \texttt{1}, \texttt{2}, the six palindromes
of the form $bcb$, and the six palindromes of the form $abcba$. It is easy to check that it contains no other palindrome, i.e., no $cabcbac$.
To show that $\gamma(\eta^\omega(\texttt{0}))$ is $\tfrac{41}{22}^+$-free, we consider the morphic word
$g(h^\omega(\texttt{0}))$ defined by the following morphisms $h$ and $g$.
\begin{align*}
    h(\texttt{0}) &= \texttt{01213012} & g(\texttt{0}) &= \texttt{0102012}\\
    h(\texttt{1}) &= \texttt{31} & g(\texttt{1}) &= \texttt{0212}\\
    h(\texttt{2}) &= \texttt{01201312} & g(\texttt{2}) &= \texttt{0121012}\\
    h(\texttt{3}) &= \texttt{0121301312} & g(\texttt{3}) &= \texttt{01020121012}
\end{align*}
We check that $g(h^\omega(\texttt{0}))$ avoids $F_\gamma$. We will show in Section~\ref{sec:ce} that $g(h^\omega(\texttt{0}))$ is $\tfrac{41}{22}^+$-free.
Since $g(h^\omega(\texttt{0}))$ is square-free and avoids $F_\gamma$, we see that $g(h^\omega(\texttt{0}))$ is $\gamma$-good. So by Lemma~\ref{ggood}, the words
$g(h^\omega(\texttt{0}))$ and $\gamma(\eta^\omega(\texttt{0}))$ have the same factor set.\footnote{Notice that $g(h^\omega(\texttt{0}))$ is not bi-infinite,
as required by Lemma~\ref{ggood}. However, $g(h^\omega(\texttt{0}))$ is uniformly recurrent, which is somehow a stronger condition for our purpose.}
Thus, $\gamma(\eta^\omega(\texttt{0}))$ is also $\tfrac{41}{22}^+$-free.

Let us prove \Cref{it:pos}.
We say that a ternary word (finite or infinite) is 16-good if it is $\tfrac{52}{27}$-free and contains at most $16$ palindromes.
We construct the set $S_{16}^{36}$ defined as follows:
a word $v$ is in $S_{16}^{36}$ if and only if there exists a 16-good word $pvs$ such that $|p|=|v|=|s|=36$.
Then we construct the Rauzy graph $G_{16}^{36}$ based on $S_{16}^{36}$, that is, such that vertices correspond to factors of length 35 and arcs correspond to factors of length 36.
We notice that $G_{16}^{36}$ consists of two connected components that are images of each other with respect to reversal.
Moreover, the connected component containing the factor \texttt{0120} is equal to the Rauzy graph
of the factors of length 36 of $\gamma(\eta^\omega(\texttt{0}))$.
Let $w$ be a bi-infinite 16-good word belonging to this connected component, that is, $w$ contains \texttt{0120}.
Since $\max_{f\in F_\gamma}|f|=36$, we see that $w$ is $\gamma$-good. By Lemma~\ref{ggood}, we conclude that $w$ has the same factor set as $\gamma(\eta^\omega(\texttt{0}))$.
By symmetry, this proves \Cref{it:pos}.

Let us prove \Cref{it:neg}.
We say that a ternary word (finite or infinite) is 17-good if it is $\tfrac{25}{13}$-free and contains at most $17$ palindromes.
Let $w$ be a recurrent 17-good word.
We check by backtracking that no infinite 17-good word avoids $\texttt{010}$.
By symmetry, $w$ must contain the six palindromes of the form $bcb$ and thus $abcba$.
So $w$ contains the 16 palindromes of $\gamma(\eta^\omega(\texttt{0}))$.
Notice that $w$ cannot contain a palindrome of the form $abacaba$ since $abacaba$
cannot be extended into a ternary square-free word.
So, without loss of generality, we assume that the potential 17th palindrome of $w$ is \texttt{0120210}.
We say that a ternary word (finite or infinite) is 17-great if it is 17-good and every palindrome it contains is a factor of \texttt{0120210}, \texttt{01210}, \texttt{02120}, \texttt{10201}, \texttt{20102}, or \texttt{21012}.
Thus $w$ is 17-great. We construct the set $S_{17}^{36}$ defined as follows:
a word $v$ is in $S_{17}^{36}$ if and only if there exists a 17-great word $pvs$ such that $|p|=|v|=|s|=36$.
Then we construct the Rauzy graph $G_{17}^{36}$ based on $S_{17}^{36}$.
Since $w$ is recurrent, $w$ is in a strongly connected component of $G_{17}^{36}$.
We have checked that the graph induced by each strongly connected component of $G_{17}^{36}$
is isomorphic to $G_{16}^{36}$.
Similarly, this shows that $w$ has the same factor set as either
$\gamma(\eta^\omega(\texttt{0}))$ or $\gamma(\eta^\omega(\texttt{0}))^R$.
\end{proof}

\section{Critical exponent of \texorpdfstring{$g(h^{\omega}({\texttt 0}))$}{g(h omega(0))}}\label{sec:ce}
Let us recall the definition of the morphisms $h$ and $g$:
\begin{align*}
    h(\texttt{0}) &= \texttt{01213012} & g(\texttt{0}) &= \texttt{0102012}\\
    h(\texttt{1}) &= \texttt{31} & g(\texttt{1}) &= \texttt{0212}\\
    h(\texttt{2}) &= \texttt{01201312} & g(\texttt{2}) &= \texttt{0121012}\\
    h(\texttt{3}) &= \texttt{0121301312} & g(\texttt{3}) &= \texttt{01020121012}
\end{align*}
In order to show that the critical exponent of $g(h^{\omega}({\texttt 0}))$ equals $\frac{41}{22}$, we will describe bispecial factors of $g(h^{\omega}({\texttt 0}))$ and their shortest return words. 
We will then apply the following result.

\begin{theorem}[\cite{DDP21}, Theorem 3]\label{thm:FormulaForE}
Let $\mathbf u$ be a uniformly recurrent aperiodic infinite word.
Let $(w_n)_{n\in\mathbb N}$ be the sequence of all bispecial factors in $\mathbf u$ ordered by length.
For every $n \in \mathbb N$, let $r_n$ be the shortest return word to the bispecial factor $w_n$ in $\mathbf u$.
Then
$$
E(\mathbf u) = 1 + \sup\acc{\frac{|w_n|}{|r_n|} \ : \ n \in \mathbb N} .
$$
\end{theorem}
Since $h$ is a primitive morphism, its fixed point is uniformly recurrent. Hence, $g(h^{\omega}({\texttt 0}))$ is uniformly recurrent, too.
We will see in the sequel that $g(h^{\omega}({\texttt 0}))$ contains infinitely many bispecial factors, thus it is aperiodic.

We will proceed in two steps.  First we determine the bispecial factors and their shortest return words in the fixed point $h^{\omega}({\texttt 0})$, and then we do the same thing in the word $g(h^{\omega}({\texttt 0}))$. 

\subsection{Bispecial factors in \texorpdfstring{$h^{\omega}({\texttt 0})$}{h omega(0)}}
To describe the structure of all bispecial factors, we use a method of~\cite{Klouda2012}.
This method does not work directly with bispecial factors, but with bispecial triplets $((a,b), w, (c,d))$, where $(a, b)$ and $(c, d)$ are unordered pairs of left and right extensions (respectively) of $w$ with $a\not =b, c\not =d$, and both $awc$ and $bwd$ or both $awd$ and $bwc$ are factors.\footnote{In general, $a,b, c, d$ can be words longer than one, but in the case of the morphism $h$ it suffices to consider letters.}
Obviously, each bispecial factor is a middle element of a bispecial triplet. We say that the bispecial factor is \textit{associated} with the bispecial triplet. A~bispecial factor may be associated with more than one bispecial triplet if it has more than two left or two right extensions.

The method of Klouda says that there is a mapping that maps each bispecial triplet to another, and that all bispecial triplets are obtained by repetitive application of the mapping to a finite set of \textit{initial} bispecial triplets.

The mapping, called the $f$\textit{-image} by Klouda, is defined as follows: The $f$-image of a~bispecial triplet $((a,b), w, (c,d))$ is the bispecial triplet $((a',b'), u_1h(w)u_2, (c',d'))$, where $u_1$ is the longest common suffix of $h(a)$ and $h(b)$, $a'u_1$ is a suffix of $h(a)$, and $b'u_1$ is a suffix of $h(b)$, and similarly, $u_2$ is the longest common prefix of $h(c)$ and $h(d)$, $u_2c'$ is a prefix of $h(c)$, and $u_2d'$ is a prefix of $h(d)$. 

\cite{Klouda2012} proved that the set of initial bispecial triplets corresponds to the set of bispecial factors without a~synchronization point.
In our case, the initial bispecial factors are:
\[
\varepsilon, \tt 1, 3, 01, 12, 13, 31, 012, 1201.
\]
In order to get all bispecial factors we do not have to consider all initial bispecial triplets since some of them have the same $f$-image. 
\begin{example}
$(({\tt 0,1}), \varepsilon, ({\tt 1,2}))$ is a bispecial triplet as $\tt 01$ and $\tt 12$ are factors of $h^{\omega}({\texttt 0})$. 
Its $f$-image is the initial triplet $(({\tt 2,1}), \varepsilon, ({\tt 3,0}))$, hence to find all bispecial factors we do not have to take $(({\tt 0,1}), \varepsilon, ({\tt 1,2}))$ into account.
\end{example}
Let us provide a list of initial bispecial triplets (giving rise to all bispecial factors) together with their first four $f$-images.

\begin{enumerate}
\item $(({\tt 0,2}), \varepsilon, ({\tt 1,3}))$
    \begin{align*}
        f\text{-image:} \  & (({\tt 0,3}), {\tt 12}, ({\tt 0,3})) \\
        f^2\text{-image:} \ & (({\tt 0,3}), {\tt 12}h({\tt 12}){\tt 0121301}, ({\tt 2,3})) \\
        f^3\text{-image:} \ & (({\tt 0,3}), {\tt 12}h({\tt 12})h^2({\tt 12})h({\tt 0121301}){\tt 012}, ({\tt 0,1}))\\
        f^4\text{-image:} \ & (({\tt 0,3}), {\tt 12}h({\tt 12})h^2({\tt 12})h^3({\tt 12})h^2({\tt 0121301})h({\tt 012}), ({\tt 0,3}))
    \end{align*}
\item $(({\tt 1,2}), \varepsilon, ({\tt 2,3}))$
\begin{align*}
        f\text{-image:} \ & {\tt ((1,2), 012, (0,1))} \\
        f^2\text{-image:} \ & {\tt ((1,2)}, h({\tt 012}), {\tt (0,3))} \\
        f^3\text{-image:} \ & (({\tt 1,2}), h^2(\tt 012)0121301, (2,3))\\
        f^4\text{-image:} \ & (({\tt 1,2}), h^3({\tt 012})h(\tt 0121301)012, (0,1))
    \end{align*}
\item $(({\tt 1,2}), \varepsilon, (\tt 0,3))$
\begin{align*}
        f\text{-image:} \ &\tt  ((1,2), 0121301, (2,3)) \\
        f^2\text{-image:} \ & (({\tt 1,2}), h(\tt 0121301)012, (0,1)) \\
        f^3\text{-image:} \ & (({\tt 1,2}), h^2({\tt 0121301})h(\tt 012), (0,3))\\
        f^4\text{-image:} \ & (({\tt 1,2}), h^3({\tt 0121301})h^2({\tt 012})\tt 0121301, (2,3))
    \end{align*}
\item $(({\tt 2,3}), \varepsilon, \tt (0,3))$
\begin{align*}
        f\text{-image:} \ &\tt  ((2,3), 013120121301, (2,3)) \\
        f^2\text{-image:} \ & (({\tt 2,3}), {\tt 01312}h(\tt 013120121301)012, (0,1)) \\
        f^3\text{-image:} \ & (({\tt 2,3}), {\tt 01312}h({\tt 01312})h^2({\tt 013120121301})h(\tt 012), (0,3))\\
        f^4\text{-image:} \ & (({\tt 2,3}), {\tt 01312}h({\tt 01312})h^2({\tt 01312})h^3({\tt 013120121301})h^2({\tt 012})\tt 0121301, (2,3))
    \end{align*}
\item $(({\tt 2,3}), \varepsilon,\tt  (1,3))$
\begin{align*}
        f\text{-image:} \ & \tt ((2,3), 01312, (0,3)) \\
        f^2\text{-image:} \ & (({\tt 2,3}), {\tt 01312}h(\tt 01312)0121301, (2,3)) \\
        f^3\text{-image:} \ & (({\tt 2,3}), {\tt 01312}h({\tt 01312})h^2({\tt 01312})h(\tt 0121301)012, (0,1))\\
        f^4\text{-image:} \ & (({\tt 2,3}), {\tt 01312}h({\tt 01312})h^2({\tt 01312})h^3({\tt 01312})h^2({\tt 0121301})h(\tt 012), (0,3))
    \end{align*}
\item $(({\tt 0,3}), \tt 1, (0,2))$
\begin{align*}
        f\text{-image:} \ & (({\tt 0,3}), {\tt 12}h(\tt 1)012, (0,1)) \\
        f^2\text{-image:} \ & (({\tt 0,3}), {\tt 12}h({\tt 12})h^2({\tt 1})h(\tt 012), (0,3)) \\
        f^3\text{-image:} \ & (({\tt 0,3}), {\tt 12}h({\tt 12})h^2({\tt 12})h^3({\tt 1})h^2({\tt 012})\tt 0121301, (2,3))\\
        f^4\text{-image:} \ & (({\tt 0,3}), {\tt 12}h({\tt 12})h^2({\tt 12})h^3({\tt 12})h^4({\tt 1})h^3({\tt 012})h(\tt 0121301)012, (0,1))
    \end{align*}
\item $(({\tt 2,3}), \tt 1, (0,3))$
\begin{align*}
        f\text{-image:} \ & (({\tt 2,3}), {\tt 01312}h(\tt 1)0121301, (2,3)) \\
        f^2\text{-image:} \ & (({\tt 2,3}), {\tt 01312}h({\tt 01312})h^2({\tt 1})h(\tt 0121301)012, (0,1)) \\
        f^3\text{-image:} \ & (({\tt 2,3}), {\tt 01312}h({\tt 01312})h^2({\tt 01312})h^3({\tt 1})h^2({\tt 0121301})h(\tt 012), (0,3))\\
        f^4\text{-image:} \ & (({\tt 2,3}), {\tt 01312}h({\tt 01312})h^2({\tt 01312})h^3({\tt 01312})h^4({\tt 1})h^3({\tt 0121301})h^2({\tt 012})\tt 0121301, (2,3))
    \end{align*}
\item $\tt ((0,3), 1, (0,3))$
\begin{align*}
        f\text{-image:} \ & (({\tt 0,3}), {\tt 12}h(\tt 1)0121301, (2,3)) \\
        f^2\text{-image:} \ & (({\tt 0,3}), {\tt 12}h({\tt 12})h^2({\tt 1})h(\tt 0121301)012, (0,1)) \\
        f^3\text{-image:} \ & (({\tt 0,3}), {\tt 12}h({\tt 12})h^2({\tt 12})h^3({\tt 1})h^2({\tt 0121301})h(\tt 012), (0,3))\\
        f^4\text{-image:} \ & (({\tt 0,3}), {\tt 12}h({\tt 12})h^2({\tt 12})h^3({\tt 12})h^4({\tt 1})h^3({\tt 0121301})h^2({\tt 012})\tt 0121301, (2,3))
    \end{align*}
\item $\tt ((2,3), 1, (2,3))$
\begin{align*}
        f\text{-image:} \ & (({\tt 2,3}), {\tt 01312}h(\tt 1)012, (0,1)) \\
        f^2\text{-image:} \ & (({\tt 2,3}), {\tt 01312}h({\tt 01312})h^2({\tt 1})h(\tt 012), (0,3)) \\
        f^3\text{-image:} \ & (({\tt 2,3}), {\tt 01312}h({\tt 01312})h^2({\tt 01312})h^3({\tt 1})h^2({\tt 012})\tt 0121301, (2,3))\\
        f^4\text{-image:} \ & (({\tt 2,3}), {\tt 01312}h({\tt 01312})h^2({\tt 01312})h^3({\tt 01312})h^4({\tt 1})h^3({\tt 012})h({\tt 0121301})\tt 012, (0,1))
    \end{align*}
\item $\tt ((1,2), 3, (0,1))$
\begin{align*}
        f\text{-image:} \ & (({\tt 1,2}), h(\tt 3), (0,3)) \\
        f^2\text{-image:} \ & (({\tt 1,2}), h^2(\tt 3)0121301, (2,3)) \\
        f^3\text{-image:} \ & (({\tt 1,2}), h^3({\tt 3})h(\tt 0121301)012, (0,1))\\
        f^4\text{-image:} \ & (({\tt 1,2}), h^4({\tt 3})h^2({\tt 0121301})h(\tt 012), (0,3))
    \end{align*}
\item $\tt ((2,3), 01, (2,3))$
\begin{align*}
        f\text{-image:} \ & (({\tt 2,3}), {\tt 01312}h(\tt 01)012, (0,1)) \\
        f^2\text{-image:} \ & (({\tt 2,3}), {\tt 01312}h({\tt 01312})h^2({\tt 01})h(\tt 012), (0,3)) \\
        f^3\text{-image:} \ & (({\tt 2,3}), {\tt 01312}h({\tt 01312})h^2({\tt 01312})h^3({\tt 01})h^2({\tt 012})\tt 0121301, (2,3))\\
        f^4\text{-image:} \ & (({\tt 2,3}), {\tt 01312}h({\tt 01312})h^2({\tt 01312})h^3({\tt 01312})h^4({\tt 01})h^3({\tt 012})h(\tt 0121301)012, (0,1))
    \end{align*}
\item $\tt ((1,2), 01, (2,3))$
\begin{align*}
        f\text{-image:} \ & (({\tt 1,2}), h(\tt 01)012, (0,1)) \\
        f^2\text{-image:} \ & (({\tt 1,2}), h^2({\tt 01})h(\tt 012), (0,3)) \\
        f^3\text{-image:} \ & (({\tt 1,2}), h^3({\tt 01})h^2(\tt 012)0121301, (2,3))\\
        f^4\text{-image:} \ & (({\tt 1,2}), h^4({\tt 01})h^3({\tt 012})h(\tt 0121301)012, (0,1))
    \end{align*}
\item $\tt ((0,3), 12, (0,1))$
\begin{align*}
        f\text{-image:} \ & (({\tt 0,3}), {\tt 12}h(\tt 12), (0,3)) \\
        f^2\text{-image:} \ & (({\tt 0,3}), {\tt 12}h({\tt 12})h^2(\tt 12)0121301, (2,3)) \\
        f^3\text{-image:} \ & (({\tt 0,3}), {\tt 12}h({\tt 12})h^2({\tt 12})h^3({\tt 12})h(\tt 0121301)012, (0,1))\\
        f^4\text{-image:} \ & (({\tt 0,3}), {\tt 12}h({\tt 12})h^2({\tt 12})h^3({\tt 12})h^4({\tt 12})h^2({\tt 0121301})h(\tt 012), (0,3))
    \end{align*}
\item $\tt ((0,2), 13, (0,1))$
\begin{align*}
        f\text{-image:} \ & (({\tt 0,3}), {\tt 12}h(\tt 13), (0,3)) \\
        f^2\text{-image:} \ & (({\tt 0,3}), {\tt 12}h({\tt 12})h^2(\tt 13)0121301, (2,3)) \\
        f^3\text{-image:} \ & (({\tt 0,3}), {\tt 12}h({\tt 12})h^2({\tt 12})h^3({\tt 13})h(\tt 0121301)012, (0,1))\\
        f^4\text{-image:} \ & (({\tt 0,3}), {\tt 12}h({\tt 12})h^2({\tt 12})h^3({\tt 12})h^4({\tt 13})h^2({\tt 0121301})h(\tt 012), (0,3))
    \end{align*}
\item $\tt ((1,2), 31, (0,2))$
\begin{align*}
        f\text{-image:} \ & (({\tt 1,2}), h(\tt 31)012, (0,1)) \\
        f^2\text{-image:} \ & (({\tt 1,2}), h^2({\tt 31})h(\tt 012), (0,3)) \\
        f^3\text{-image:} \ & (({\tt 1,2}), h^3({\tt 31})h^2(\tt 012)0121301, (2,3))\\
        f^4\text{-image:} \ & (({\tt 1,2}), h^4({\tt 31})h^3({\tt 012})h(\tt 0121301)012, (0,1))
    \end{align*}
\item $\tt ((2,3), 012, (1,3))$
\begin{align*}
        f\text{-image:} \ & (({\tt 2,3}), {\tt 01312}h(\tt 012), (0,3)) \\
        f^2\text{-image:} \ & (({\tt 2,3}), {\tt 01312}h({\tt 01312})h^2(\tt 012)0121301, (2,3)) \\
        f^3\text{-image:} \ & (({\tt 2,3}), {\tt 01312}h({\tt 01312})h^2({\tt 01312})h^3({\tt 012})h(\tt 0121301)012, (0,1))\\
        f^4\text{-image:} \ & (({\tt 2,3}), {\tt 01312}h({\tt 01312})h^2({\tt 01312})h^3({\tt 01312})h^4({\tt 012})h^2({\tt 0121301})h(\tt 012), (0,3))
    \end{align*}
\item $\tt ((1,3), 012, (0,3))$
\begin{align*}
        f\text{-image:} \ & (({\tt 1,2}), h(\tt 012)0121301, (2,3)) \\
        f^2\text{-image:} \ & (({\tt 1,2}), h^2({\tt 012})h(\tt 0121301)012, (0,1)) \\
        f^3\text{-image:} \ & (({\tt 1,2}), h^3({\tt 012})h^2({\tt 0121301})h(\tt 012), (0,3))\\
        f^4\text{-image:} \ & (({\tt 1,2}),h^4({\tt 012})h^3({\tt 0121301})h^2({\tt 012})\tt 0121301, (2,3))
    \end{align*}
\item $\tt ((0,3), 1201, (2,3))$
\begin{align*}
        f\text{-image:} \ & (({\tt 0,3}), {\tt 12}h(\tt 1201)012, (0,1)) \\
        f^2\text{-image:} \ & (({\tt 0,3}), {\tt 12}h({\tt 12})h^2({\tt 1201})h(\tt 012), (0,3)) \\
        f^3\text{-image:} \ & (({\tt 0,3}),  {\tt 12}h({\tt 12})h^2({\tt 12})h^3({\tt 1201})h^2(\tt 012)0121301, (2,3))\\
        f^4\text{-image:} \ & (({\tt 0,3}), {\tt 12}h({\tt 12})h^2({\tt 12})h^3({\tt 12})h^4({\tt 1201})h^3({\tt 012})h(\tt 0121301)012, (0,1))
    \end{align*}
\end{enumerate}
Observing Items 1 to 18, we can see that the $f^{n}$-image and $f^{n+3}$-image of any given initial bispecial triplet have the same left and right extensions. Therefore, we may write down recurrence relations and their explicit solutions for the Parikh vectors of bispecial factors $w_{\ell+3n}$ associated with the $f^{\ell+3n}$-image of each initial bispecial triplet, where $\ell \in \{1,2,3\}$. For this purpose, we will let $M$ denote the matrix of the morphism $h$, i.e., 
$$M=\left(\begin{array}{rrrr} 
2&0&2&2\\
3&1&3&4\\
2&0&2&2\\
1&1&1&2
\end{array}\right)\,.$$

We will limit our considerations to the cases $1$ and $15$; the other cases are analogous and are left to the reader.
\begin{description}
\item[Case 1]  We have $\vec w_1=\left(\begin{smallmatrix} 
0\\1\\1\\0
\end{smallmatrix}\right)$ and the recurrence relations read, for each $n \in \mathbb N$,
\begin{align*}
\vec w_{4+3n}&=\left(\begin{smallmatrix} 
0\\1\\1\\0
\end{smallmatrix}\right)+M\left(\begin{smallmatrix} 
0\\1\\1\\0
\end{smallmatrix}\right)+M^2\left(\begin{smallmatrix} 
0\\1\\1\\0
\end{smallmatrix}\right)+M^3\vec w_{1+3n}+M^2\left(\begin{smallmatrix} 
2\\3\\1\\1
\end{smallmatrix}\right)+M\left(\begin{smallmatrix} 
1\\1\\1\\0
\end{smallmatrix}\right)\,,\\
\vec w_{2+3n}&=\left(\begin{smallmatrix} 
0\\1\\1\\0
\end{smallmatrix}\right)+M\vec w_{1+3n}+\left(\begin{smallmatrix} 
2\\3\\1\\1
\end{smallmatrix}\right)\,,\\
\vec w_{3+3n}&=\left(\begin{smallmatrix} 
0\\1\\1\\0
\end{smallmatrix}\right)+M\vec w_{2+3n}+\left(\begin{smallmatrix} 
1\\1\\1\\0
\end{smallmatrix}\right)\,.
\end{align*}
The explicit solution, for each $n \in \mathbb N$, is given by
$$\vec w_{1+3n}= \sum_{j=0}^{3n}M^j\left(\begin{smallmatrix} 
0\\1\\1\\0
\end{smallmatrix}\right)+\sum_{j=0}^{n-1}M^{3j}\left( M^2\left(\begin{smallmatrix} 
2\\3\\1\\1
\end{smallmatrix}\right)+M\left(\begin{smallmatrix} 
1\\1\\1\\0
\end{smallmatrix}\right)\right)\,.$$

\item[Case 15] We have $\vec w_1=M\left(\begin{smallmatrix} 
0\\1\\0\\1
\end{smallmatrix}\right)+\left(\begin{smallmatrix} 
1\\1\\1\\0
\end{smallmatrix}\right)$ and the recurrence relations read, for each $n \in \mathbb N$,
\begin{align*}
\vec w_{4+3n}&= M^3\vec w_{1+3n}+M\left(\begin{smallmatrix} 
2\\3\\1\\1
\end{smallmatrix}\right)+\left(\begin{smallmatrix} 
1\\1\\1\\0
\end{smallmatrix}\right)\,,\\
\vec w_{2+3n}&=M\vec w_{1+3n}\,,\\
\vec w_{3+3n}&=M\vec w_{2+3n}+\left(\begin{smallmatrix} 
2\\3\\1\\1
\end{smallmatrix}\right)\,.
\end{align*}
The explicit solution, for each $n \in \mathbb N$, is given by
$$\vec w_{1+3n}= M^{3n+1}\left(\begin{smallmatrix} 
0\\1\\0\\1
\end{smallmatrix}\right)+\sum_{j=0}^{n}M^{3j}\left(\begin{smallmatrix} 
1\\1\\1\\0
\end{smallmatrix}\right)+\sum_{j=0}^{n-1} M^{3j+1}\left(\begin{smallmatrix} 
2\\3\\1\\1
\end{smallmatrix}\right)\,.$$
\end{description}
\subsection{Shortest return words to bispecial factors in \texorpdfstring{$h^{\omega}(\tt 0)$}{h omega(0)}}
To find the shortest return words to bispecial factors, we will use the following statement. For words $r$ and $s$, we write $\vec r\le \vec{s}$ if the $i$-th component of $\vec{r}$ is less than or equal to the $i$-th component of $\vec s$ for all $i$.
\begin{proposition}\label{prop:retwords}
If a bispecial factor $w$ in $h^{\omega}({\normaltt 0})$ is associated to a bispecial triplet 
\[
((a,b), u_1h(v)u_2, (c,d)),
\]
where 
\begin{itemize}
    \item $u_1$ is the longest common suffix of $h(a)$ and $h(b)$ and $u_2$ is the longest common prefix of $h(c)$ and $h(d)$,
    \item $h(v)$ has the synchronization points $\bullet h(v)\bullet$, and 
    \item $v$ has only two left and two right extensions,
\end{itemize} then the Parikh vectors of return words to $w$ are the same as the Parikh vectors of images by $h$ of return words to $v$.

In particular, if $r$ is a return word to $v$ such that $\vec r\le \vec s$ for any return word $s$ to $v$, then $\vv{h(r)}$ is the Parikh vector of the shortest return word to $w$.
\end{proposition}
\begin{proof}
Let $z$ be a return word to $w=u_1h(v)u_2$, so that we have
$$zw=\underbrace{u_1h(v)u_2\cdots}_{z}u_1h(v)u_2.$$
Since $v$ has only two left and two right extensions, using the definition of bispecial triplet, $h(v)$ is always preceded by $u_1$ and followed by $u_2$, hence $h(v)$ occurs only twice in $u_1h(v)u_2\cdots u_1h(v)u_2$. 
Knowing the synchronization points $\bullet h(v)\bullet u_2\cdots u_1\bullet h(v)\bullet$ and by the injectivity of $h$, there exists a unique return word $r$ to $v$ such that $h(v)u_2\cdots u_1h(v)=h(r)h(v)$. Again, since $h(v)$ is always preceded by $u_1$ and followed by $u_2$, we get  $u_1h(r)h(v)u_2=u_1h(r)u_1^{-1}u_1h(v)u_2$, consequently, $z=u_1h(r)u_1^{-1}$ and $\vec z=\vv{h(r)}$.

The second statement is then obvious. 
\end{proof}

\begin{observation}
The words $h({\tt 0}), h({\tt 2}), h({\tt 3}), h(i{\tt 1}), h({\tt 1}i)$, where $i \in \{\tt 0,2,3\}$, have synchronization points $\bullet h({\tt 0})\bullet, \bullet h({\tt 2})\bullet, \bullet h({\tt 3})\bullet, \bullet h(i{\tt 1})\bullet, \bullet h({\tt 1}i)\bullet$.
\end{observation}

By Proposition~\ref{prop:retwords}, it suffices to find the shortest return words to the shortest bispecial factors having only two left and two right extensions, in each case 1 to 18. 
By the shortest return word we mean a return word $r$ satisfying $\vec r \le \vec s$ for any other return word. 
Here is the complete list:

\begin{enumerate}
\item ${\tt 12}h({\tt 12})h({\tt 01})\tt 01213013$ is the shortest return word to ${\tt 12}h(\tt 12)0121301$
\item $h(\tt 01231)$ is the shortest return word to $h(\tt 012)$
\item $\tt 0121301231$ is the shortest return word to $\tt 0121301$
\item $\tt 013120121301231012$ is the shortest return word to $\tt 013120121301$
\item $\tt 0131201213$ is the shortest return word to $\tt 01312$
\item ${\tt 12}h(\tt 1)012013$ is the shortest return word to ${\tt 12}h(\tt 1)012$
\item ${\tt 01312}h(\tt 1)0121301231012$ is the shortest return word to ${\tt 01312}h(\tt 1)0121301$
\item ${\tt 12}h(\tt 1)0121301231012013$ is the shortest return word to ${\tt 12}h(\tt 1)0121301$
\item ${\tt 01312}h(\tt 1)0120131201213$ is the shortest return word to ${\tt 01312}h(\tt 1)012$
\item $\tt 301$ and $\tt 312$ are the shortest return words to $\tt 3$
\item ${\tt 01312}h(\tt 01)012130131231012$ is the shortest return word to ${\tt 01312}h(\tt 01)012$
\item $h(\tt 01)01201312$ is the shortest return word to $h(\tt 01)012$
\item ${\tt 12}h(\tt 12)012130$ is the shortest return word to ${\tt 12}h(\tt 12)$
\item $\tt 130$ is the shortest return word to $\tt 13$
\item $\tt 312$ is the shortest return word to $\tt 31$
\item ${\tt 01312}h(\tt 012)3101213$ is the shortest return word to ${\tt 01312}h(\tt 012)$
\item $h(\tt 01231)$ is the shortest return word to $h(\tt 012)0121301$
\item $\tt 12013$ is the shortest return word to $\tt 1201$
\end{enumerate}
Applying Proposition~\ref{prop:retwords}, we get the Parikh vectors of the shortest return words to all sufficiently long bispecial factors, leaving only a few short bispecial factors to consider separately.
Let us list the shortest return words to all bispecial factors in cases $1$ and $15$.
\begin{description}
\item[Case 1:] $\vec r_2=M\left(\begin{smallmatrix}1\\2\\1\\1\end{smallmatrix}\right)$ is the Parikh vector of the shortest return word to $w_2={\tt 12}h(\tt 12)0121301$. By Proposition~\ref{prop:retwords}, $\vec r_n=M^{n-2}\vec r_2$ is the Parikh vector of the shortest return word to $w_n$ for $n\ge 3$.
\item[Case 15:] $r_0=\tt 312$ is the shortest return word to $w_0=\tt 31$. By Proposition~\ref{prop:retwords}, $\vec r_n=M^{n}\vec r_0$ is the Parikh vector of the shortest return word to $w_n$ for $n\ge 1$.
\end{description}

\subsection{Bispecial factors in \texorpdfstring{$g(h^{\omega}(\tt 0))$}{g(h omega(0))}} 
Let us describe how to get all sufficiently long bispecial factors in $g(h^{\omega}(\tt 0))$ using our knowledge of bispecial triplets in $h^{\omega}(\tt 0)$.
\begin{proposition}\label{prop:retwordS_g}
Let $w$ be a bispecial factor in $g(h^{\omega}({\tt 0}))$ such that it has at least two synchronization points. Then there exists a unique bispecial factor $v$ in $h^{\omega}({\tt 0})$ such that $w=xg(v)y$, where $x$ is the longest common suffix of $g(a)$ and $g(b)$ for some $a,b \in \{\tt 0,1,12,3\}$, $a\not =b$, and $y$ is the longest common prefix of $g(c)$ and $g(d)$ for some $c,d \in \{\tt 01,1,2,3\}$, $c\not =d$, and either $avc$ and $bvd$ or $avd$ and $bvc$ are factors of $h^{\omega}(\tt 0)$. 

Moreover, if $v$ has only two left and two right extensions, then the Parikh vectors of return words to $w$ are the same as the Parikh vectors of images by $g$ of return words to $v$.

In particular, if $r$ is a return word to $v$ such that $\vec r\le \vec s$ for any return word $s$ to $v$, then $\vv{g(r)}$ is the Parikh vector of the shortest return word to $w$.
\end{proposition}
\begin{proof}
Let us highlight the first and the last synchronization point of $w$ by writing $w=x\bullet \cdots \bullet y$.  Then by injectivity of $g$, there exists a unique factor $v$ in $h^{\omega}(\tt 0)$ such that $w=xg(v)y$. Now, $v$ is bispecial, otherwise we have a contradiction with the choice of synchronization points.
Since $g(\tt 0)$ is a prefix of $g(\tt 3)$ and $g(\tt 2)$ is a suffix of $g(\tt 3)$, in order to determine $x$ and $y$ we need to consider $\tt 12$ instead of $\tt 2$ among the left extensions of $v$ (this is without loss of generality since $\tt 2$ is always preceded by $\tt 1$ in $h^{\omega}({\tt 0})$). Similarly, we have to consider $\tt 01$ instead of $\tt 0$ among the right extensions of $v$. The statement on how to get $x$ and $y$ is then straightforward.

The remaining part is analogous to the proof of Proposition~\ref{prop:retwords}.
\end{proof}

Using Proposition~\ref{prop:retwordS_g}, to get the Parikh vectors of all bispecial factors and their shortest return words in $g(h^{\omega}(\tt 0))$, we have to treat three cases:
\begin{enumerate}
\item[a)] Bispecial factors in $g(h^{\omega}(\tt 0))$ having at most one synchronization point.
\item[b)] Bispecial factors in $g(h^{\omega}(\tt 0))$ having at least two synchronization points arising from bispecial factors in $h^{\omega}(\tt 0)$ with more than two left or right extensions.
\item[c)] Bispecial factors in $g(h^{\omega}(\tt 0))$ having at least two synchronization points arising from bispecial factors in $h^{\omega}(\tt 0)$ with exactly two left and two right extensions.
\end{enumerate}

\noindent \textbf{Case a)} Table~\ref{table:a} contains the list of all bispecial factors in $g(h^{\omega}(\tt 0))$ having at most one synchronization point together with their shortest return words. (There are more of them in some cases, but we always list only one.)

\begin{table}[ht]
\centering
\begin{tabular}{l|l}
    \text{bispecial factor} & \text{shortest return word} \\
    \hline 
$\tt 0$ & $\tt 01$\\
$\tt 1$ & $\tt 12$\\
$\tt 2$ & $\tt 20$\\
$\tt 10$ & $\tt 10120$\\
$\tt 01$ & $\tt 012$\\
$\tt 02$ & $\tt 02012$\\
$\tt 20$ & $\tt 201$\\
$\tt 12$ & $\tt 120$\\
$\tt 21$ & $\tt 21201$\\
$\tt 201$ & $\tt 2010$\\
$\tt 012$ & $\tt 0121$\\
$\tt 120$ & $\tt 1202$\\
$\tt 2012$ & $\tt 2012021$\\
$\tt 0120$ & $\tt 0120102$\\
$\tt 1201$ & $\tt 1201210$\\
$\tt 201210120$ & $\tt 20121012010$\\
$\tt 120102012$ & $\tt 12010201210$
\end{tabular}
\caption{List of bispecial factors and their shortest return words in case a).}\label{table:a}
\end{table}

\noindent \textbf{Case b)} Table~\ref{table:b} contains the list of all bispecial factors in $g(h^{\omega}(\tt 0))$ arising from bispecial factors in $h^{\omega}(\tt 0)$ having more than two left or right extensions (i.e., arising from $\tt 1,01,12,012$), together with their shortest return words.  (There are more of them in some cases, but we always list only one.)

\begin{table}[ht]
\centering
\begin{tabular}{l|l}
    \text{bispecial factor} & \text{shortest return word} \\
    \hline 
${\tt 012}g(\tt 1)01$ & ${\tt 012}g(\tt 1)0102$\\
${\tt 012}g(\tt 1)0102012$ & ${\tt 012}g(\tt 1)010201210120102$\\
${\tt 20121012}g(\tt 1)01$ & ${\tt 20121012}g(\tt 1)0102012021$\\
${\tt 20121012}g(\tt 1)0102012$ & ${\tt 20121012}g(1)\tt 0102012021$\\
${\tt 12}g(\tt 01)01$ & ${\tt 12}g(\tt 01)01210$\\
${\tt 20121012}g(\tt 01)01$ & ${\tt 20121012}g(\tt 01)01210120212010$ \\
${\tt 012}g(\tt 12)0$ & ${\tt 012}g(\tt 12)0102$\\
${\tt 012}g(\tt 12)0102012$ & ${\tt 012}g(\tt 12)0102012021201020121$ \\
${\tt 12}g(\tt 012)0$ & ${\tt 12}g(\tt 012)0102012101202$\\
${\tt 20121012}g(\tt 012)0$ & ${\tt 20121012}g(\tt 012)0212010$ \\
${\tt 12}g(\tt 012)0102012$ & ${\tt 12}g(\tt 012)0102012101202$
\end{tabular}
\caption{List of bispecial factors and their shortest return words in case b).}\label{table:b}
\end{table}

\noindent \textbf{Case c)} We will list the Parikh vectors of bispecial factors and their shortest return words coming only from cases $1$ and $15$ in $h^{\omega}(\tt 0)$. Let us denote by $W_n$ the bispecial factor in $g(h^{\omega}(\tt 0))$ arising from $w_n$ in $h^{\omega}(\tt 0)$. Moreover, we will use the matrix $N$ of the morphism $g$, i.e., 
$$N=\left(\begin{array}{rrrr}
3&1&2&4\\
2&1&3&4\\
2&2&2&3
\end{array}\right)\,.$$
By Proposition~\ref{prop:retwordS_g}, in all cases, the shortest return word to $W_n$ has Parikh vector equal to $N\vv{r}_n$, where $r_n$ is the shortest return word to $w_n$ in $h^{\omega}(\tt 0)$.
\begin{description}
\item[Case 1] For all $n \in \mathbb N$, we have
\begin{align*}
W_{1+3n}&={\tt 012}g(w_{1+3n}){\tt 0102012} \ \text{ when } n \ge 1\,,\\
W_{2+3n}&={\tt 012}g(w_{2+3n}){\tt 01}\,, \text{ and} \\
W_{3+3n}&={\tt 012}g(w_{3+3n}){\tt 0}\,. \\
\end{align*}


\item[Case 15] For all $n \in \mathbb N$, we have
\begin{align*}
W_0&={\tt 12}g({\tt 31}){\tt 01}\,,\\
W_{1+3n}&={\tt 12}g(w_{1+3n}){\tt 0}\,,\\
W_{2+3n}&={\tt 12}g(w_{2+3n}){\tt 0102012}\,, \text{ and}\\
W_{3+3n}&={\tt 12}g(w_{3+3n}){\tt 01}\,. \\
\end{align*}
\end{description}

We have prepared everything to be able to compute the lengths of all bispecial factors in $g(h^{\omega}(\tt 0))$ together with the lengths of their shortest return words. Thus, using Theorem~\ref{thm:FormulaForE} we will be able to determine the critical exponent of $g(h^{\omega}(\tt 0))$.
\begin{theorem}\label{thm:E_g}
The critical exponent of the word $g(h^{\omega}(\tt 0))$ equals $\frac{41}{22}$.
\end{theorem}
\begin{proof}
According to Theorem~\ref{thm:FormulaForE}, we have to show that for each bispecial factor $W$ and the shortest return word $R$ to $W$ in $g(h^{\omega}(\tt 0))$, the ratio $\frac{|W|}{|R|}$ is smaller than or equal to $\frac{19}{22}$. We have divided bispecial factors into three classes corresponding to the cases a), b), and c) described above.
\begin{enumerate}
\item[a)] For all bispecial factors in the first class, the ratio $\frac{|W|}{|R|}$ is at most $\frac{9}{11}<\frac{19}{22}$, see Table~\ref{table:a}.
\item[b)] For all bispecial factors in the second class, the ratio $\frac{|W|}{|R|}$ is at most $\frac{19}{22}$, see Table~\ref{table:b}. Moreover, the exponent $\frac{41}{22}$ is reached by the factor $${\tt 20121012}g({\tt 1}){\tt 010201202120121012}g({\tt 1}){\tt 0102012}=({\tt 2012101202120102012021})^{\frac{41}{22}}\,.$$
\item[c)] In this last class, we will treat only the bispecial factors described in cases $1$ and $15$.
To compute the lengths of bispecial factors and their shortest return words, we diagonalize the matrix $M$. Using standard linear algebra, we obtain $M=XDX^{-1}$, where 
\[
\arraycolsep=4pt
D=\left(\begin{array}{rrrr}
6&0&0&0\\
0&1&0&0\\
0&0&0&0\\
0&0&0&0
\end{array}\right), \ X=\left(\begin{array}{rrrr}
1&2&1&1\\
2&-1&1&0\\
1&2&0&-1\\
1&-3&-1&0
\end{array}\right), \  X^{-1}=\frac{1}{15}\left(\begin{array}{rrrr}
4&1&4&5\\
3&-3&3&0\\
-5&10&-5&-10\\
10&-5&-5&5
\end{array}\right).
\]
We will repeatedly use the formula
$$(1,1,1)NM^j=(1,1,1)NXD^jX^{-1}=\frac{1}{5}(44\cdot 6^j-9, 11\cdot 6^j+9, 44\cdot 6^j-9, 55\cdot 6^j)\,. $$

\noindent
\textbf{Case 1:} For all $n \in \mathbb N$ with $n\ge 1$, we have
\begin{align*}
\dfrac{|W_{1+3n}|}{|R_{1+3n}|}
&=\dfrac{(1,1,1)\left(\left(\begin{smallmatrix} 1\\1\\1
\end{smallmatrix}\right)+N\vv{w}_{1+3n}+\left(\begin{smallmatrix} 3\\2\\2
\end{smallmatrix}\right)\right)}{(1,1,1)NM^{3n-1}\vv{r}_2}\\
&=\dfrac{(1,1,1)\left(\left(\begin{smallmatrix} 4\\3\\3
\end{smallmatrix}\right)+N\left(\displaystyle\sum_{j=0}^{3n}M^j\left(\begin{smallmatrix} 0\\1\\1\\0
\end{smallmatrix}\right)+\sum_{j=0}^{n-1}M^{3j+2}\left(\begin{smallmatrix} 2\\3\\1\\1
\end{smallmatrix}\right)+\sum_{j=0}^{n-1}M^{3j+1}\left(\begin{smallmatrix} 1\\1\\1\\0
\end{smallmatrix}\right)\right)\right)}{(1,1,1)NM^{3n}\left(\begin{smallmatrix} 1\\2\\1\\1
\end{smallmatrix}\right)}\\
&=\dfrac{10+11\sum_{j=0}^{3n}6^j+44\sum_{j=0}^{n-1}6^{3j+2}+\frac{1}{5}\sum_{j=0}^{n-1}(99\cdot 6^{3j+1}-9)}{33 \cdot 6^{3n}}\\
&=\dfrac{10+\frac{11}{5}(6^{3n+1}-1)+\frac{44\cdot 36}{215}(6^{3n}-1)+\frac{99\cdot 6}{5\cdot 215}(6^{3n}-1)-\frac{9}{5}n}{33 \cdot 6^{3n}}\\
&<\dfrac{688}{1075}+\dfrac{10}{33\cdot 6^{3n}}<\dfrac{19}{22}\,. \\
\end{align*}
For all $n \in \mathbb N$, we have
\begin{align*}
\dfrac{|W_{2+3n}|}{|R_{2+3n}|}
&=\dfrac{(1,1,1)\left(\left(\begin{smallmatrix} 1\\1\\1
\end{smallmatrix}\right)+N\vv{w}_{2+3n}+\left(\begin{smallmatrix} 1\\1\\0
\end{smallmatrix}\right)\right)}{(1,1,1)NM^{3n}\vv r_2}\\
&=\dfrac{(1,1,1)\left(\left(\begin{smallmatrix} 2\\2\\1
\end{smallmatrix}\right)+N\left(\displaystyle\sum_{j=0}^{3n+1}M^j\left(\begin{smallmatrix} 0\\1\\1\\0
\end{smallmatrix}\right)+\sum_{j=0}^{n}M^{3j}\left(\begin{smallmatrix} 2\\3\\1\\1
\end{smallmatrix}\right)+\sum_{j=0}^{n-1}M^{3j+2}\left(\begin{smallmatrix} 1\\1\\1\\0
\end{smallmatrix}\right)\right)\right)}{(1,1,1)NM^{3n+1}\left(\begin{smallmatrix} 1\\2\\1\\1
\end{smallmatrix}\right)}\\
&=\dfrac{5+11\sum_{j=0}^{3n+1}6^j+44\sum_{j=0}^{n}6^{3j}+\frac{1}{5}\sum_{j=0}^{n-1}(99\cdot 6^{3j+2}-9)}{33 \cdot 6^{3n+1}}\\
&=\dfrac{5+\frac{11}{5}(6^{3n+2}-1)+\frac{44}{215}(6^{3n+3}-1)+\frac{99\cdot 36}{5\cdot 215}(6^{3n}-1)-\frac{9}{5}n}{33 \cdot 6^{3n+1}}\\
&<\dfrac{688}{1075}+\dfrac{5}{33\cdot 6^{3n+1}}<\dfrac{19}{22}\,. \\
\end{align*}
For all $n \in \mathbb N$, we have
\begin{align*}
\dfrac{|W_{3+3n}|}{|R_{3+3n}|}
&=\dfrac{(1,1,1)\left(\left(\begin{smallmatrix} 1\\1\\1
\end{smallmatrix}\right)+N\vv w_{3+3n}+\left(\begin{smallmatrix} 1\\0\\0
\end{smallmatrix}\right)\right)}{(1,1,1)NM^{3n+1}\vv r_2}\\
&=\dfrac{(1,1,1)\left(\left(\begin{smallmatrix} 2\\1\\1
\end{smallmatrix}\right)+N\left(\displaystyle\sum_{j=0}^{3n+2}M^j\left(\begin{smallmatrix} 0\\1\\1\\0
\end{smallmatrix}\right)+\sum_{j=0}^{n}M^{3j+1}\left(\begin{smallmatrix} 2\\3\\1\\1
\end{smallmatrix}\right)+\sum_{j=0}^{n}M^{3j}\left(\begin{smallmatrix} 1\\1\\1\\0
\end{smallmatrix}\right)\right)\right)}{(1,1,1)NM^{3n+2}\left(\begin{smallmatrix} 1\\2\\1\\1
\end{smallmatrix}\right)}\\
&=\dfrac{4+11\sum_{j=0}^{3n+2}6^j+44\sum_{j=0}^{n}6^{3j+1}+\frac{1}{5}\sum_{j=0}^{n}(99\cdot 6^{3j}-9)}{33 \cdot 6^{3n+2}}\\
&=\dfrac{4+\frac{11}{5}(6^{3n+3}-1)+\frac{44\cdot 6}{215}(6^{3n+3}-1)+\frac{99}{5\cdot 215}(6^{3n+3}-1)-\frac{9}{5}(n+1)}{33 \cdot 6^{3n+2}}\\
&<\dfrac{688}{1075}+\dfrac{4}{33\cdot 6^{3n+2}}<\dfrac{19}{22}\,.
\end{align*}

\noindent
\textbf{Case 15:} Let us start by noting that $\dfrac{|W_0|}{|R_0|}=\dfrac{(1,1,1)\left(\left(\begin{smallmatrix} 0\\1\\1
\end{smallmatrix}\right)+N\left(\begin{smallmatrix} 0\\1\\0\\1
\end{smallmatrix}\right)+\left(\begin{smallmatrix} 1\\1\\0
\end{smallmatrix}\right)\right)}{(1,1,1)N\left(\begin{smallmatrix} 0\\1\\1\\1
\end{smallmatrix}\right)}=\dfrac{19}{22}$. Thus we have found another factor of exponent $\frac{41}{22}$:
$${\tt 12}g({\tt 31}){\tt 0121012}g({\tt 31}){\tt 01}=({\tt 1201020121012021201210})^{\frac{41}{22}}\,.$$
For all $n \in \mathbb N$, we have
\begin{align*}
\dfrac{|W_{1+3n}|}{|R_{1+3n}|}
&=\dfrac{(1,1,1)\left(\left(\begin{smallmatrix} 0\\1\\1
\end{smallmatrix}\right)+N\vv w_{1+3n}+\left(\begin{smallmatrix} 1\\0\\0
\end{smallmatrix}\right)\right)}{(1,1,1)NM^{3n+1}\vv r_0}\\
&=\dfrac{(1,1,1)\left(\left(\begin{smallmatrix} 1\\1\\1
\end{smallmatrix}\right)+N\left(M^{3n+1}\left(\begin{smallmatrix} 0\\1\\0\\1
\end{smallmatrix}\right)+\displaystyle\sum_{j=0}^{n}M^{3j}\left(\begin{smallmatrix} 1\\1\\1\\0
\end{smallmatrix}\right)+\sum_{j=0}^{n-1}M^{3j+1}\left(\begin{smallmatrix} 2\\3\\1\\1
\end{smallmatrix}\right)\right)\right)}{(1,1,1)NM^{3n+1}\left(\begin{smallmatrix} 0\\1\\1\\1
\end{smallmatrix}\right)}\\
&=\dfrac{3+\frac{66\cdot 6^{3n+1}+9}{5}+\frac{1}{5}\sum_{j=0}^{n}(99\cdot6^{3j}-9)+44\sum_{j=0}^{n-1}6^{3j+1}}{22 \cdot 6^{3n+1}}\\
&<\dfrac{817}{1075}+\dfrac{2}{55\cdot 6^{3n}}<\dfrac{19}{22}\,.
\end{align*}
For all $n \in \mathbb N$, we have
\begin{align*}
\dfrac{|W_{2+3n}|}{|R_{2+3n}|}
&=\dfrac{(1,1,1)\left(\left(\begin{smallmatrix} 0\\1\\1
\end{smallmatrix}\right)+N\vv w_{2+3n}+\left(\begin{smallmatrix} 3\\2\\2
\end{smallmatrix}\right)\right)}{(1,1,1)NM^{3n+2}\vv r_0}\\
&=\dfrac{(1,1,1)\left(\left(\begin{smallmatrix} 3\\3\\3
\end{smallmatrix}\right)+N\left(M^{3n+2}\left(\begin{smallmatrix} 0\\1\\0\\1
\end{smallmatrix}\right)+\displaystyle\sum_{j=0}^{n}M^{3j+1}\left(\begin{smallmatrix} 1\\1\\1\\0
\end{smallmatrix}\right)+\sum_{j=0}^{n-1}M^{3j+2}\left(\begin{smallmatrix} 2\\3\\1\\1
\end{smallmatrix}\right)\right)\right)}{(1,1,1)NM^{3n+2}\left(\begin{smallmatrix} 0\\1\\1\\1
\end{smallmatrix}\right)}\\
&=\dfrac{9+\frac{66\cdot 6^{3n+2}+9}{5}+\frac{1}{5}\sum_{j=0}^{n}(99\cdot6^{3j+1}-9)+44\sum_{j=0}^{n-1}6^{3j+2}}{22 \cdot 6^{3n+2}}\\
&<\dfrac{817}{1075}+\dfrac{9}{110\cdot 6^{3n+1}}<\dfrac{19}{22}\,. 
\end{align*}
For all $n \in \mathbb N$, we have
\begin{align*}
\dfrac{|W_{3+3n}|}{|R_{3+3n}|}
&=\dfrac{(1,1,1)\left(\left(\begin{smallmatrix} 0\\1\\1
\end{smallmatrix}\right)+N\vv w_{3+3n}+\left(\begin{smallmatrix} 1\\1\\0
\end{smallmatrix}\right)\right)}{(1,1,1)NM^{3n+3}\vv r_0}\\
&=\dfrac{(1,1,1)\left(\left(\begin{smallmatrix} 1\\2\\1
\end{smallmatrix}\right)+N\left(M^{3n+3}\left(\begin{smallmatrix} 0\\1\\0\\1
\end{smallmatrix}\right)+\displaystyle\sum_{j=0}^{n}M^{3j+2}\left(\begin{smallmatrix} 1\\1\\1\\0
\end{smallmatrix}\right)+\sum_{j=0}^{n}M^{3j}\left(\begin{smallmatrix} 2\\3\\1\\1
\end{smallmatrix}\right)\right)\right)}{(1,1,1)NM^{3n+3}\left(\begin{smallmatrix} 0\\1\\1\\1
\end{smallmatrix}\right)}\\
&=\dfrac{4+\frac{66\cdot 6^{3n+3}+9}{5}+\frac{1}{5}\sum_{j=0}^{n}(99\cdot6^{3j+2}-9)+44\sum_{j=0}^{n}6^{3j}}{22 \cdot 6^{3n+3}}\\
&<\dfrac{817}{1075}+\dfrac{29}{110\cdot 6^{3n+3}}<\dfrac{19}{22}\,. \qedhere
\end{align*}
\end{enumerate}
\end{proof}

\section{Overpals and letter patterns}\label{sec:related}
The notions of overpals and letter patterns defined in the preliminaries happen to be related to square-free words with few palindromes.
\begin{lemma}\label{equiv}
Let $w$ be a bi-infinite ternary square-free word. Then the following properties are equivalent:
\begin{enumerate}[label=\normalfont(\alph*)]
\item $w$ contains at most 16 palindromes.\label{eq_16}
\item $w$ avoids overpals.\label{eq_overpals}
\item $w$ avoids the letter pattern $abcacba$.\label{eq_abcacba}
\end{enumerate}
\end{lemma}

\begin{proof}
First, notice that since $w$ is square-free, the length of every non-empty palindrome in $w$ is odd. Also, $w$ contains the palindromic letter pattern $aba$ if and only if
$w$ contains $cabac$.

\Cref{eq_16} $\Longrightarrow$ \Cref{eq_overpals}:
A computer check shows that there are only finitely many ternary square-free words avoiding \texttt{010} and containing at most 16 palindromes.
So, if $w$ contains at most 16 palindromes, then $w$ must contain all 6 factors of the form $aba$.
This implies that $w$ contains the following 16 palindromes: $\varepsilon$, \texttt{0}, \texttt{1}, \texttt{2}, six palindromes of the form $aba$, and six palindromes of the form $abcba$.
Thus $w$ contains no other palindrome. In particular, $w$ avoids overpals,
since an overpal is a palindrome such that the first letter occurs at least three times.

\Cref{eq_overpals} $\Longrightarrow$ \Cref{eq_abcacba}:
The letter pattern $abcacba$ is an overpal $axax^Ra$ with $x=bc$.
So if $w$ avoids overpals, then $w$ avoids $abcacba$.

\Cref{eq_abcacba} $\Longrightarrow$ \Cref{eq_16}:
Notice that $w$ avoids $cbcacbc$, since $cbcacbc$ has no square-free extension.
If $w$ also avoids $abcacba$,
then the palindrome $bcacb$ cannot be extended to a larger palindrome.
Thus, $w$ contains no palindrome other than the 16 palindromes mentioned above.
\end{proof}

Using Lemma~\ref{equiv}, we obtain the analogs of Theorems~\ref{thm:exp_pairs}\ref{e16},~\ref{41_22}\ref{it:16p}, and~\ref{41_22}\ref{it:pos},
where ``containing (at most) 16 palindromes'' is replaced by ``avoiding overpals'' or ``avoiding the letter pattern $abcacba$''.
This proves Conjecture 17 of~\cite{ARS2017} that there exists an infinite ternary $\tfrac{41}{22}^+$-free word that avoids overpals.
This also complements results by~\cite{C2016} about words avoiding $abcacba$ with respect to the critical exponent and the factor complexity.

\medskip

\cite{P2016} considered the other letter patterns that are minimally avoidable by infinite ternary square-free words, namely $abaca$, $abcab$, and $abacbc$.
We give simpler proofs of her results as well as the minimal critical exponent
$\tfrac74^+$ for words avoiding $abacbc$.
Here, contrary to the case of $abcacba$, there is no critical exponent such that the factor complexity is polynomial.

\begin{theorem}\label{LPexp}
There exist exponentially many ternary words that are:
\begin{enumerate}[label=\normalfont(\alph*)]
\item $\tfrac{15}8^+$-free and avoid $abaca$.
\item $\tfrac{11}6^+$-free and avoid $abcab$.
\item $\tfrac74^+$-free and avoid $abacbc$.
\end{enumerate}
\end{theorem}

\begin{proof}
In order to verify that the proposed words below avoid the associated letter pattern $p$, note that we 
only need to check the factors up to length $|p|$.
\begin{enumerate}[label=(\alph*)]
\item Applying the following 72-uniform morphism to any ternary $\tfrac74^+$-free word gives a ternary $\tfrac{15}8^+$-free word avoiding $abaca$.
{\footnotesize
\begin{align*}
\texttt{0}&\rightarrow\texttt{010210120102120121020120210120102101202102012102120102101201021201210212} \\
\texttt{1}&\rightarrow\texttt{010210120210201202101201021012021020121021201210201202101201021201210212} \\
\texttt{2}&\rightarrow\texttt{010210120210201210212012102012021012010210120210201202101201021201210212} 
\end{align*}
}

\item Applying the following 73-uniform morphism to any ternary $\tfrac74^+$-free word gives a ternary $\tfrac{11}6^+$-free word avoiding $abcab$.
{\footnotesize
\begin{align*}
\texttt{0}&\rightarrow\texttt{0121020102120210120212010212021012102012101202120102120210120212010201210} \\
\texttt{1}&\rightarrow\texttt{1202101210201021201020121020102120210120212010201210201021201020121012021} \\
\texttt{2}&\rightarrow\texttt{2010212021012102012101202101210201021201020121012021012102012101202120102} 
\end{align*}
}

\item Applying the following 128-uniform morphism to any 4-ary $\tfrac75^+$-free word gives a ternary $\tfrac74^+$-free word avoiding $abacbc$.
{\footnotesize
\begin{align*}
\texttt{0}\rightarrow &\ p\texttt{120102101210212012101201020120210201021012102120210201202120121021202102}\\
&\ \texttt{010210121}\\
\texttt{1}\rightarrow &\ p\texttt{120102101210212021020102101201020120210201021012102120121012010210121021}\\
&\ \texttt{202102012}\\
\texttt{2}\rightarrow &\ p\texttt{120102101210212021020120212012102120210201021012010201202102010210121021}\\
&\ \texttt{202102012}\\
\texttt{3}\rightarrow &\ p\texttt{212021020102101210212012101201021012102120210201021012010201202120121021}\\
&\ \texttt{202102012} 
\end{align*}
}
where {\footnotesize$p=\texttt{02120121012010201202102010210120102012021201210}$}. \qedhere
\end{enumerate}
\end{proof}

\acknowledgements
We thank the anonymous referees for comments which helped to improve the paper.

\nocite{*}
\bibliographystyle{abbrvnat}
\bibliography{biblio-dmtcs}
\label{sec:biblio}

\end{document}